\documentclass[reqno]{amsart}
\usepackage{macros/preprintStyle}
\pdfoutput=1 %

\usepackage{amssymb}

\usepackage{bm}
\usepackage{comment}
\usepackage{adjustbox} 

\usepackage{mathtools}

\usepackage{algorithmic}

\usepackage{xspace}

\usepackage{graphicx}
\usepackage{color}
\usepackage{circuitikz}
\usepackage{tikz}
\usetikzlibrary{calc,positioning,shapes}
\usetikzlibrary{patterns,decorations.pathmorphing,decorations.markings}
\usetikzlibrary{external}

\usepackage{pgfplots}
\pgfplotsset{compat=newest}
\usepackage[margin=10pt,font=small,labelfont=bf,labelsep=endash]{caption}
\usepackage{subcaption}

\usetikzlibrary{intersections, backgrounds}
\usetikzlibrary{circuits}
\usetikzlibrary{circuits.ee.IEC}
\usepackage{pgfplots}
\pgfplotsset{compat=newest}

\definecolor{color0}{rgb}{0.12156862745098,0.466666666666667,0.705882352941177}
\definecolor{color1}{rgb}{1,0.498039215686275,0.0549019607843137}
\definecolor{color2}{rgb}{0.172549019607843,0.627450980392157,0.172549019607843}
\definecolor{color3}{rgb}{0.83921568627451,0.152941176470588,0.156862745098039}
\definecolor{color4}{rgb}{0.580392156862745,0.403921568627451,0.741176470588235}
\definecolor{color5}{rgb}{0,0,0}

\definecolor{mycolor1}{rgb}{0.00000,0.44700,0.74100}
\definecolor{mycolor2}{rgb}{0.85000,0.32500,0.09800}
\definecolor{mycolor3}{rgb}{0.92900,0.69400,0.12500}
\definecolor{mycolor4}{rgb}{0.46600,0.67400,0.18800}
\definecolor{mycolor5}{rgb}{0.49400,0.18400,0.55600}

\newcommand{\smoothFunctions}[3][]{\ifthenelse{\equal{#1}{}}{\mathcal{C}^{#2}}{\mathcal{C}_{#1}^{#2}}(#3)}

\newcommand{\dist}[1][]{\ifthenelse{\equal{#1}{}}{\mathbb{D}}{#1_{\mathbb{D}}}}

\DeclareMathOperator{\rank}{rank}

\DeclareMathOperator{\sspan}{span}

\DeclareMathOperator{\tr}{tr}

\renewcommand{\dist}{\mathrm{dist}}

\newcommand{\calP}{\mathcal{P}}

\newcommand{\rdim}{r} 

\newcommand{\N}{\ensuremath\mathbb{N}}

\newcommand{\R}{\ensuremath\mathbb{R}}
\newcommand{\C}{\ensuremath\mathbb{C}}

\newcommand{\ImgUnit}{\mathrm{i}}

\newcommand{\IdxHam}{p}
\newcommand{\IdxDis}{q}
\newcommand{\NrHam}{P}
\newcommand{\NrDis}{Q}

\newcommand{\pDim}{d}

\newcommand{\prmtr}{\mu}
\newcommand{\prmtrSet}{\calP}

\newcommand{\fdim}{N} 
\newcommand{\Lindbladian}{\mathcal{L}} 
\newcommand{\SteadyState}{\rho_{\mathrm{ss}}} 

\newcommand{\SolSet}{\mathcal{M}}

\newcommand{\FHinteraction}{g} 

\title[MOR for parametric Lindbladian]{Model Order Reduction for Parametric Dissipative Quantum Systems}
\author{Bernard Haasdonk${}^\dagger$, Benjamin Stamm${}^\dagger$  \and Zhuoyao Zeng${}^\dagger$}

\address{${}^{\dagger}$ Institute of Applied Analysis and Numerical Simulation, University of Stuttgart, Pfaffenwaldring 57, 70569 Stuttgart, Germany}
\email{\{given\_name.surname\}@mathematik.uni-stuttgart.de}

\date{\today}

\begin{document}

\begin{abstract}
We consider reduced basis approximations 
of non-equilibrium steady states 
of parameter-dependent Lindbladians, 
whose ambient matrix space grows 
exponentially with system size. 
We give a unified structural view 
of the problem through its characterizations
as an eigenvalue problem or
as a trace-constrained source problem. 
Building on these characterizations, 
we propose a residual-driven reduced basis framework 
for both formulations. 
Under the assumption of a unique steady state, 
we derive residual-based state-error estimates, 
and establish (sub-)exponential convergence 
of the greedy algorithms
under real-analytic parameter dependence of the problem data. 
Numerical experiments on boundary-driven Fermi--Hubbard 
and dissipative XYZ chains support the analysis and 
suggest that the trace-constrained greedy method 
offers the best overall balance of accuracy 
and computational cost.
\end{abstract}

\maketitle



\section{Introduction}\label{sec:intro}

We consider finite-dimensional open quantum systems with underlying complex Hilbert-space of dimension~$\fdim$. 
Their states are represented by \emph{density matrices} 
$\rho\in\C^{\fdim\times\fdim}$,
i.e., $\rho^\dagger=\rho,$
$\tr\rho=1$ and 
$\rho\succeq0.$
Here, $\rho\succeq0$ means that $\rho$ is positive semidefinite (psd).

Interactions with an external environment generally induce non-unitary dynamics. 
Following \cite{GoriniKossakowskiSudarshan1976,Lindblad1976}, 
the evolution under  the time-independent Markovian setting is governed 
by the \emph{Lindblad master equation} $\dot{\rho}(t)
\,=\,
\Lindbladian\rho(t),$
where the operator $\Lindbladian$ is called a \emph{Lindbladian}. 
We are concerned with its \emph{non-equilibrium steady states} $\SteadyState$, 
i.e., density matrices  satisfying $\Lindbladian\SteadyState\,=\,0.$
These steady states describe stationary configurations and, 
whenever the dynamics converges to a steady state, 
determine the long-time values of physical observables. 
Hence, steady state computation is a central task for open quantum systems, 
including driven-dissipative many-body models \cite{WeimerKshetrimayumOrus2021}.

We study a parametric family of Lindbladians $\bigl(\Lindbladian(\prmtr)\bigr)_{\prmtr\in\prmtrSet}$, 
where $\prmtrSet\subseteq\R^\pDim$ is compact. 
Repeated steady-state computations arise, 
for example, in phase-diagram exploration, 
sensitivity analysis, 
and the evaluation of parameter-dependent observables. 
For example, consider a many-body system with $L$ sites, 
each associated with a local Hilbert space $\C^s$. 
Note that $s=2$
for a chain of qubits. 
The full Hilbert space $(\C^s)^{\otimes L}$ then has dimension $\fdim=s^L$, 
and a density matrix acting on it is an $\fdim\times\fdim$ Hermitian matrix. 
The real vector space of such Hermitian matrices has dimension 
$\fdim^2=s^{2L}$, and imposing the unit-trace condition leaves 
$\fdim^2-1$ independent real parameters. 
This exponential growth makes repeated high-fidelity solves costly 
and motivates a reduced model for the associated many-query-many-body problem.

Reduced basis methods (RBM) provide a natural model order reduction (MOR) framework 
for such many-query problems \cite{HesthavenRozzaStamm2016, Haasdonk2017ReducedBasis}. 
They approximate the parametric solution set in a low-dimensional space generated by selected high-fidelity solutions, 
or \emph{snapshots}, 
and separate the expensive construction of this space from subsequent low-dimensional solves. 
In a greedy construction, new snapshots are usually selected by maximizing an error indicator over a training set. 
If the parametric Lindbladian admits the affine decomposition \eqref{eq:Lindbladian_affine_decomposition}, 
the quantities required for this sweep can be assembled from precomputed reduced components; 
see \Cref{subsec:reduced_operator_assembly}.

Our goal is to construct low-dimensional RB spaces by residual-driven greedy sampling, 
enabling fast repeated approximation of the parametric steady states, 
and to establish error estimates and convergence guarantees for the resulting RB procedures.
There are two natural ways to formulate the reduced steady-state approximation problem. 
One may regard a steady state as an eigenvector 
associated to the eigenvalue $0$ of the Lindbladian, 
or determine it directly by minimizing the Lindbladian residual subject to the unit-trace constraint. 
To make these two viewpoints concrete, 
for selected snapshot parameters $\prmtr_1,\ldots,\prmtr_J$, 
let
$V 
\, \coloneqq \, 
\operatorname{span}_{\R}
\left\{
\SteadyState(\prmtr_1),\ldots,\SteadyState(\prmtr_J)
\right\}
$
be the resulting reduced space
in the real vector space of Hermitian matrices,
equipped with the Frobenius norm~$\|\cdot\|$. 
The first viewpoint 
defines a norm-constrained residual minimization problem
\begin{equation} \label{eq:abstract_EVP}
\min_{X\in V}
\| \Lindbladian(\prmtr) X \|^2
\qquad
\text{subject to}
\qquad
\|X\|=1,
\end{equation}
followed by trace normalization. 
\eqref{eq:abstract_EVP} can be interpreted as an eigenvalue problem 
for the psd operator $\Lindbladian(\prmtr)^\dagger \Lindbladian(\prmtr)$
restricted to $V$. 
The second determines the approximation directly 
through the trace-constrained residual minimization problem
\begin{equation}
\min_{X\in V}
\| \Lindbladian(\prmtr)X \|^2
\qquad
\text{subject to}
\qquad
\tr X=1.
\label{eq:abstract_constraint_res_min}
\end{equation}

Hermiticity, as a linear constraint, is inherited from the snapshot solution
space.  The trace-constrained formulation enforces unit trace directly,
whereas the eigenvector approximation is normalized to unit trace whenever
its trace is nonzero.
However, positive semidefiniteness is not automatic and may be imposed by optional post-processing.
Despite their different constraints, 
both formulations fit into a common residual-minimization framework. 
This perspective allows us to develop the corresponding greedy procedures and analyze their approximation errors and convergence in parallel. 
Before that, 
we first review the relevant literature.

\subsection{Related work} \label{subsec:lit_review}

Early reduced basis (RB) techniques were developed 
for parameter-dependent structural problems \cite{NoorPeters1980}. 
The certified RB framework for parametrized PDEs combines snapshot approximation, 
a-posteriori error estimation, 
greedy sampling, 
and offline--online decomposition \cite{PrudhommeEtAl2002JFE,PruRVP02,MadayPateraTurinici2002,RozzaHuynhPatera2008}; 
see \cite{HesthavenRozzaStamm2016,QuarteroniManzoniNegri2016,Haasdonk2017ReducedBasis} for systematic presentations. 
The relation between greedy approximation and Kolmogorov widths underlying our convergence analysis 
was developed in \cite{BinevEtAl2011,BuffaEtAl2012,DeVorePetrovaWojtaszczyk2013}.

RB techniques have also been developed for parametric eigenvalue problems, 
including certified bounds for symmetric positive definite problems \cite{MaYOPR00}, 
asymptotic a-posteriori estimation for the first elliptic eigenpair \cite{FumagalliManzoniParoliniVerani2016}, 
simultaneous approximation of several eigenvalues and parameter-dependent multiplicities \cite{HWD17}, 
subspace acceleration and guaranteed bounds for the smallest eigenvalue \cite{SirK16}, 
clusters and eigenvalue intersections \cite{BHP24}, 
and non-symmetric problems \cite{ConjungoTaumhasEtAl2024}. 
Recently, a certified construction for the smallest eigenspace 
of general parametric Hermitian matrices is given in \cite{ManucciStammZeng2026},
and a general local higher order RB construction is studied in \cite{StammZeng2026}. 
In quantum many-body physics, 
related snapshot-projection ideas appear as {eigenvector continuation} \cite{FrameEtAl2018,DuguetEtAl2024} 
and in RB surrogates for parametric spin Hamiltonians \cite{Herbst2022,Brehmer2023}.

For open quantum systems, 
established methods for a single steady state or trajectory include sparse iterative solvers \cite{NationEtAl2015}, 
matrix-product representations \cite{ProsenZnidari2009,CuiCiracBanuls2015}, 
variational principles \cite{Weimer2015}, 
and neural-network approaches \cite{VicentiniEtAl2019}; 
see \cite{WeimerKshetrimayumOrus2021} for a review. 
MOR techniques have also been applied to open-system dynamics through balanced truncation \cite{SchaeferBungEtAl2011}, 
low-rank approximation of the evolving density operator \cite{LeBrisRouchon2013}, 
and exact structure-preserving reductions tailored to prescribed initial states or observables \cite{GrigolettoEtAl2025}. 
These works primarily reduce a fixed system, 
whereas we seek a reusable surrogate for the steady states of an entire parametric family.

The work most closely related to the present study is \cite{ChristiansenBaranPaaske2025}, which applies RB ideas to driven-dissipative Markovian systems. Its steady-state approximation uses the smallest-eigenvalue formulation for $\Lindbladian(\prmtr)^\dagger\Lindbladian(\prmtr)$, cf.~\eqref{eq:abstract_EVP}, together with a residual-driven greedy construction.

\subsection{Main contributions} \label{subsec:main_contributions}

Our contributions concern the reduced methodology, the underlying mathematical
structure, and its convergence analysis.

\begin{enumerate}
\item
We provide a systematic mathematical treatment of the structure underlying
the Lindbladian steady-state problem in \Cref{sec:pLindbladian}, 
exploiting its interpretations as a
parametric eigenvalue problem and a parametric source problem.  This viewpoint
organizes the roles of the kernel geometry, trace constraint, uniqueness and
stability, and leads naturally to residual-based state-error estimates for
both reduced formulations.

\item
We develop a concrete offline--online RB framework for parametric Lindbladian
steady state approximations in \Cref{sec:RBM}.  
It treats the previously proposed reduced eigenvalue problem
and a new trace-constrained residual formulation within a common
residual-driven greedy construction, including implementable reduced
assemblies.

\item
We establish a unified convergence framework for the two greedy procedures in \Cref{sec:error-analysis},
showing step by step how continuity and uniform well-posedness 
of the steady state solution map yields valid weak-greedy
selection and how analytic parameter dependence yields
(sub-)exponential convergence, including the effect of a finite training set.  
To the best of our knowledge, 
this is the first convergence-rate result
for greedy RB constructions in a parametric
eigenvalue setting.
\end{enumerate}

\subsection{General notation} \label{subsec:notation}

$ \mathcal{H} \coloneqq \{ X \in \C^{\fdim \times \fdim} : X = X^\dagger \}$ denotes the set of $\fdim \times \fdim$ Hermitian matrices.
It is viewed as a \emph{real} Hilbert space, 
equipped with the Frobenius inner product $\langle X, Y \rangle \coloneqq \tr X^\dagger Y  = \tr XY $ 
and the induced norm $\| X \| \coloneqq {\langle X, X \rangle^{1/2}}$ for $X, Y \in\mathcal{H}$.
The identity matrix of size $\fdim \times \fdim$ is denoted by $I$. 
The commutator and anti-commutator of $X, Y \in \mathcal{H}$ are defined as 
$\big[X, Y\big] \coloneqq XY - YX$ and $\big\{ X, Y \big\} \coloneqq XY + YX$, respectively.
The set of all continuous linear operators from $\mathcal{H}$ to $\mathcal{H}$ is denoted by $\mathcal{B}(\mathcal{H})$.

For $k \in \{0, 1\}$, 
we write $A_k \coloneqq \{ X \in \mathcal{H} : \tr X = k \}$. 
In particular, $A_0$ is a linear subspace of $\mathcal{H}$ consisting of zero-trace Hermitian matrices,
and $A_1$ is an affine subspace consisting of Hermitian matrices with unit trace.
The set of positive semidefinite matrices is denoted by $\mathcal{H}_+ \coloneqq \{ X \in \mathcal{H} : X \succeq 0 \}$,
and the set of quantum states is defined as $S \coloneqq \mathcal{H}_+ \cap A_1 = \{ X \in \mathcal{H} : X \succeq 0, \tr X = 1 \}$.

For $X \in \mathcal{H}$, $M \subseteq \mathcal{H}$,
the distance from $X$ to $M$ is denoted by $\dist(X,M) \coloneqq \inf_{Y \in M}\|X-Y\|$. 

\section{Parametric Lindbladians and their steady-state sets}\label{sec:pLindbladian}

We review the basic properties of Lindbladian operators and
quantum states before introducing the parametric steady-state problem and its
notion of approximability.

\subsection{Quantum states and Lindbladian operators} \label{subsec:quantum_states_Lindbladian}

In this subsection, 
let us fix a Hermitian matrix $H = H^\dagger \in \mathcal{H}$
as well as arbitrary matrices $L_1, \ldots, L_\NrDis \in \C^{\fdim \times \fdim}$, 
and we consider the Lindbladian operator in \emph{Gorini-Kossakowski-Sudarshan-Lindblad (GKSL) form} defined as
\begin{equation}\label{eq:Lindbladian_unparametric}
\Lindbladian X 
\coloneqq -\ImgUnit \big[ H , X \big] 
+ 
\sum_{\IdxDis=1}^\NrDis \big( L_\IdxDis X  L_\IdxDis^\dagger - \frac{1}{2} \big\{ L_\IdxDis^\dagger L_\IdxDis, X\big \} \big)
\qquad \text{for } X \in \mathcal{H}. 
\end{equation}

\begin{lemma}[Basic properties of Lindbladians]
\label[Lemma]{lem:Lindbladian_properties}
The Lindbladian $\Lindbladian$ defined in
\eqref{eq:Lindbladian_unparametric} is a linear operator
$\Lindbladian:\mathcal H\to A_0$ and satisfies
$(\ker\Lindbladian)\cap S\neq\emptyset$
and 
$\ker\Lindbladian
=
\sspan_{\R}\bigl((\ker\Lindbladian)\cap S\bigr).$

In particular, the steady-state problem, 
i.e., finding $\SteadyState \in S$ such that $\Lindbladian\SteadyState = 0$,
has a solution in~$S$,
and a solution is unique if and only if $\dim\ker\Lindbladian=1$.
\end{lemma}

\begin{proof}[Proof sketch]
Linearity, Hermiticity preservation and
$\tr(\Lindbladian X)=0$ follow directly from
\eqref{eq:Lindbladian_unparametric}. 
The existence of stationary states 
and the relation $\ker\Lindbladian
=
\sspan_{\R}\bigl((\ker\Lindbladian)\cap S\bigr)$
for finite-dimensional GKSL generators
are established in~\cite[Sec.~3.6 and~4.1]{BaumgartnerNarnhofer2008}. 
The final assertion follows from $\ker\Lindbladian
=
\sspan_{\R}\bigl((\ker\Lindbladian)\cap S\bigr)$ and the trace-one
normalization; see also \cite[Sec.~3.6 and~4.1]{BaumgartnerNarnhofer2008}.
\end{proof}

\begin{lemma}[Basic geometric properties]
\label[lemma]{lem:geometric_properties}
It holds that
\[
\mathcal H=A_0\oplus\sspan_{\R}\{I\},
\qquad
A_1
=
{I}/{N}
+A_0,
\qquad
S=A_1\cap\mathcal H_+.
\]
In particular, $A_1$ is an affine convex subset of $\mathcal H$,
$S$ is a compact convex subset of $\mathcal H$, and
\[
{1}/{\sqrt N} \leq \|X\|  
\quad\text{for }X\in A_1,
\qquad
{1}/{\sqrt N}\leq\|\rho\|\leq1
\quad\text{for }\rho\in S.
\]
\end{lemma}

\begin{proof}
Since $\langle X,I\rangle=\tr X$, we have
$A_0=\sspan_{\R}\{I\}^{\perp}$, and every $X\in A_1$ has the unique
representation $X=I/N+X_0$ with $X_0\in A_0$. Moreover,
$1=\langle X,I\rangle\leq\sqrt N\|X\|$ for $X\in A_1$. If
$\lambda_1,\ldots,\lambda_N\geq0$ are the eigenvalues of $\rho\in S$, then
$\sum_i\lambda_i=1$ and
$1/N\leq\|\rho\|^2=\sum_i\lambda_i^2\leq1$. Finally,
$S=A_1\cap\mathcal H_+$ is closed and convex, and its boundedness together
with finite dimensionality gives compactness. 
\end{proof}

\begin{remark}[Trace-zero formulation of the steady-state problem]
\label[Remark]{rem:steady_state_eqn_as_LES}
By \Cref{lem:geometric_properties}, every
$\SteadyState\in S$ can be written uniquely 
as $\SteadyState=I/N+X$ with $X\in A_0$. 
Consequently,
\begin{equation}
\Lindbladian\SteadyState=0
\qquad\Longleftrightarrow\qquad
\Lindbladian X
=
- 
N^{-1}
\Lindbladian I.
\label{eq:steady_state_eqn_as_LES}
\end{equation}
Therefore, the steady-state problem on $A_1$ is
equivalently an inhomogeneous linear equation on~$A_0$. 
\end{remark}

\subsection{Parametric formulation and approximability of the steady state solution set} \label{subsec:parametric_lindbladians_sol_set}

We first specify the class of problems under consideration.

\begin{assumption}[Affine decomposition structure]
\label[assumption]{ass:affine_decomp_param_Lindbladian}
For a compact domain $\prmtrSet\subseteq\R^\pDim$, the parametric
Lindbladian $\Lindbladian:\prmtrSet\to\mathcal B(\mathcal H)$ has the form
\begin{equation}
\Lindbladian(\prmtr)X
=
\sum_{p=1}^{\NrHam}
\theta_p(\prmtr) \mathcal{K}_p(X)
+
\sum_{q=1}^{\NrDis}
\phi_q(\prmtr)\mathcal D_q(X),
\label{eq:Lindbladian_affine_decomposition}
\end{equation}
where 
\begin{align*}
\mathcal{K}_p(X)
\coloneqq \, 
-\ImgUnit \big[ H_p, X \big]
\qquad \text{and} \qquad
\mathcal D_q(X)
& \coloneqq \, 
\sum_{\ell=1}^{m_q}
\bigl(
L_{q,\ell}XL_{q,\ell}^\dagger
-
\frac12
\{L_{q,\ell}^\dagger L_{q,\ell},X\}
\bigr)
\end{align*}
for $p=1,\ldots,\NrHam$ and $q=1,\ldots,\NrDis$,
with each $H_\IdxHam\in\mathcal H$ 
and $L_{q,\ell}\in\C^{\fdim\times\fdim}$ fixed.
The coefficient functions
$\theta_p:\prmtrSet\to\R$ and
$\phi_q:\prmtrSet\to[0,\infty)$ are also fixed scalar functions.
\end{assumption}
We remark that nonnegativity of the $\phi_q$ is required for $\Lindbladian(\prmtr)$ to satisfy \eqref{eq:Lindbladian_unparametric}.

Next, we define the solution set to be approximated and specify its (linear) approximability. 

\begin{definition}\label[definition]{def:sol_set}
For $\Lindbladian: \prmtrSet \to \mathcal{B}(\mathcal{H})$ as in \Cref{ass:affine_decomp_param_Lindbladian},
its {\emph{steady state solution set}} is defined as 
\begin{align*} 
\SolSet \; \coloneqq \; \{ \rho \in S \, | \, \exists \prmtr \in \prmtrSet : \Lindbladian(\prmtr) \rho = 0 \}. 
\end{align*}
The \emph{Kolmogorov-$n$-width} of $\mathcal{M}$ is defined as 
\begin{align*}
d_n(\mathcal{M})
\;\coloneqq
\inf_{\substack{V\subseteq\mathcal H\\ \dim V\leq n}}
\sup_{\rho \in \mathcal{M}}
\mathrm{dist}(\rho,V)
\qquad \text{for } n\in \N.
\end{align*}
\end{definition}

Although it is customary to call $\mathcal M$ the "solution manifold",  
it is not necessarily a manifold in the topological sense.
Hence, we refer to it as the "steady-state solution set".

Since $\mathcal{M}\subseteq S \subseteq A_1$,
it is also possible to approximate $\mathcal{M}$ by affine subspaces. 

\begin{lemma}[Linear and affine widths]
\label[lemma]{lem:linear_affine_widths_equivalent}
For $n\geq0$, define the affine Kolmogorov-$n$-width by 
\[
d_n^{\mathrm{aff}}(\mathcal M)
\coloneqq
\inf_{\substack{V\subseteq\mathcal H,\,\dim V\leq n\\X_0\in\mathcal H}}
\sup_{\rho\in\mathcal M}\dist(\rho,X_0+V).
\]
Then it holds that 
\(
d_{n+1}(\mathcal M)
\leq d_n^{\mathrm{aff}}(\mathcal M)
\leq d_n(\mathcal M).
\)
\end{lemma}

\begin{proof}
The upper bound follows by choosing $X_0=0$. 
For the lower bound, 
$X_0+V$ is contained in $\sspan_{\R}(\{X_0\}\cup V)$, 
whose dimension is at most $n+1$. 
\end{proof}

Thus, the linear and affine widths are equivalent up to an index shift. 
We focus on the classical Kolmogorov-$n$-width; 
further properties of $\mathcal M$ and $d_n(\mathcal M)$ are discussed 
in \Cref{subsec:width-convergence}.

\section{Reduced formulations and numerical methods}
\label{sec:RBM}

The Kolmogorov widths introduced in \Cref{sec:pLindbladian} quantify 
the approximability of the steady-state solution set, 
and this section aims to provide a constructive reduced model.

We review the general reduced basis method in \Cref{subsec:RBM_strategy_general}, 
describe the offline assembly of the reduced operators in \Cref{subsec:reduced_operator_assembly}, 
formulate the two reduced problems in \Cref{subsec:RBM_formulations}, 
and discuss optional positivity post-processing procedures in \Cref{subsec:post_processing}.

\subsection{Reduced basis method}
\label{subsec:RBM_strategy_general}

The following assumption is fundamental for RBM. 

\begin{assumption}[Steady-state solver]
We assume access to a solver which, 
for every selected $\prmtr\in\prmtrSet$, 
yields a $\rho_\prmtr$, 
such that $\dist \big(\rho_\prmtr, S \cap\ker \Lindbladian(\prmtr) \big)$ is below any prescribed numerical accuracy.
We denote $\SteadyState(\prmtr) \coloneqq \rho_\prmtr$, 
and regard it as an \emph{exact steady state}. 
\end{assumption}

Uniqueness is not required for constructing the snapshot spaces or evaluating
the residual indicators, 
and $\SteadyState(\prmtr)$ denotes the representative
returned by the chosen steady-state solver.  
We only impose uniqueness later, 
when state-error estimates and convergence 
for a single-valued solution map are considered.

For clarity, let us now elaborate the RBM procedure. 

\begin{remark}[RBM for parametric Lindbladian] 
\label[remark]{rem:RBM_pLindblad_outline} 
Given a parametric Lindbladian $\Lindbladian(\prmtr)$ as in \Cref{ass:affine_decomp_param_Lindbladian},
the RB procedure consists of the following two phases: 
\begin{itemize}
\item 
The \emph{offline phase} assembles a reduced space $V$ based on snapshot solutions $\SteadyState(\prmtr_1),\dots, \SteadyState(\prmtr_J)$.
Common assembly strategies include \emph{proper orthogonal decomposition (POD)},
or \emph{greedy methods}, see \Cref{subsubsec:POD_vs_greedy}.
Once $V$ is constructed, 
with an orthonormal basis (ONB) $B_1, \dots,  B_\rdim$ available,
we compute and store parameter-independent quantities \eqref{eq:offline_HH_matrix}--\eqref{eq:offline_DD_matrix},
which will be used for reduced approximation of the steady states.
\item 
In the \emph{online phase}, given any $\prmtr \in \prmtrSet$,
we assemble the reduced model \eqref{eq:online_assembly_reduced_normal_matrix}, 
choose a reduced formulation \eqref{eq:abstract_EVP} or \eqref{eq:abstract_constraint_res_min}, 
compute a corresponding reduced solution 
$\rho_{\mathrm{eig}}^{(r)}(\prmtr)$ or $\rho_{\mathrm{aff}}^{(r)}(\prmtr)$ 
as specified later in \eqref{eq:trace_normalized_reduced_eigenvector} or \eqref{eq:reduced_affine_approximation}, respectively.
The resulting approximation is Hermitian. 
It has unit trace in the affine
formulation \eqref{eq:abstract_constraint_res_min} and, 
provided that the trace normalization is well-defined, also
in the eigenvalue formulation \eqref{eq:abstract_EVP}.
Post-processing procedures for positivity, 
as described in \Cref{subsec:post_processing}, 
can be optionally performed. 
\end{itemize}
\end{remark}

\begin{remark}[Reduced evaluation of observables]
  \label{rem:observable}
For an observable $O \in \mathcal{H}$, 
one is interested in {its expectation value with respect to the steady state $\rho(\prmtr)$},
i.e.,
\[
\langle O\rangle_{\rho(\prmtr)}
\coloneqq
\langle O,\rho(\prmtr)\rangle
=
\tr\bigl(O\rho(\prmtr)\bigr),
\]
for different parameters $\prmtr\in\prmtrSet$.
Let $V$ be an offline-constructed reduced space, 
with ONB $B_1,\ldots,B_r$. 
For $\star\in\{\mathrm{eig},\mathrm{aff}\}$, 
let $x_\star^{(r)}(\prmtr)$ denote the coefficient vector 
of $\rho_\star^{(r)}(\prmtr)$ in this basis.
Once $O$ is supplied, 
its reduced representation $o_r
\coloneqq
(
\langle O,B_1\rangle,\ldots,\langle O,B_r\rangle
)^\top$ is computed once, 
and the expectation is evaluated as
\[
\langle O\rangle_{\rho_\star^{(r)}(\prmtr)}
\, = \, 
o_r^\top x_\star^{(r)}(\prmtr).
\]
Thus, for the evaluation of the observable $O$,
only the one-time construction of~$o_r$ involves full-order quantities.
Subsequent parameter queries require no state reconstruction or other full-order computations,
unless positivity postprocessing is applied.
\end{remark}

Next, we describe two common ways to construct $V$.

\subsubsection{Subspace assembly strategies} \label[subsubsection]{subsubsec:POD_vs_greedy}

For prescribed snapshots
$\SteadyState(\prmtr_1),\ldots,\SteadyState(\prmtr_J)$ and a cut-off
$\varepsilon>0$, the \emph{POD basis} is constructed as follows.
\begin{enumerate} 
\item
Form the Gram matrix $G\in\R^{J\times J}$ with
$G_{ij}=\langle\SteadyState(\prmtr_i),\SteadyState(\prmtr_j)\rangle$.
\item
Compute eigenpairs
$Gq_i=\lambda_iq_i$ with
$\lambda_1\geq\cdots\geq\lambda_J\geq0$ and orthonormal~$q_i$.
\item
Retain the indices $i=1,\ldots,r$ satisfying
$\lambda_i>\varepsilon$ and set
\begin{equation*}
B_i^{\mathrm{POD}}
\coloneqq
\frac{1}{\sqrt{\lambda_i}}
\sum_{j=1}^J(q_i)_j \, \SteadyState(\prmtr_j).
\end{equation*}
\end{enumerate}
The space
$V^{\mathrm{POD}}=\sspan_\R\{B_1^{\mathrm{POD}},\ldots,
B_r^{\mathrm{POD}}\}$ minimizes the mean-square Frobenius projection error
of the prescribed snapshots among all $r$-dimensional subspaces.  Its
construction, however, requires all snapshot solves in advance; see, e.g.,
\cite{HesthavenRozzaStamm2016, Haasdonk2017ReducedBasis} for more discussions.

Alternatively, the \emph{greedy construction} starts from
$V_1=\sspan_\R\{I\}$ and, for a finite training set
$\Xi_{\mathrm{train}}\subseteq\prmtrSet$ and 
an \emph{error indicator} $\Delta^{(r)}:\prmtrSet\to [0,\infty)$
selects 
\begin{equation*}
  \prmtr_{\mathrm{next}}
  \in
  \arg\max_{\prmtr\in\Xi_{\mathrm{train}}}
  \Delta^{(r)}(\prmtr).
\end{equation*}
The selected steady state is computed and appended to the
current space unless the maximal indicator satisfies the stopping criterion.
Two indicators will be introduced in \Cref{subsec:RBM_formulations}, 
and the complete construction will be given in \Cref{alg:greedy_reduced_basis}.

\subsection{Offline computation of reduced operators}
\label{subsec:reduced_operator_assembly}

Let $ B_1,\ldots,B_r$ be the Frobenius-orthonormal basis of the reduced space
$V $ constructed, 
e.g. by one of the methods described in \Cref{subsubsec:POD_vs_greedy}. 
For $\prmtr\in\prmtrSet$, define the self-adjoint positive semidefinite operator
\begin{equation*}
\mathcal A(\prmtr)
\coloneqq
\Lindbladian(\prmtr)^\dagger
\Lindbladian(\prmtr)
\in\mathcal B(\mathcal H).
\end{equation*}
Its compression to $V$ is represented by the reduced matrix
$A_r(\prmtr)\in\R^{r\times r}$ with entries
\begin{equation}
\bigl(A_r(\prmtr)\bigr)_{ij}
\coloneqq
\left\langle
\mathcal A(\prmtr)B_j,
B_i
\right\rangle 
=
\left\langle
\Lindbladian(\prmtr)B_j,
\Lindbladian(\prmtr)B_i
\right\rangle 
\qquad \text{for }
i,j=1,\dots,r
.
\label{eq:reduced_normal_matrix}
\end{equation}
For any parameter dependent coefficients $x(\prmtr)=\big(x_j(\prmtr) \big)_{j=1}^r$ 
and $\rho^{(r)}(\prmtr) \coloneqq \sum_{j=1}^r x_j(\prmtr) B_j$,
we have 
\begin{equation*}
\bigl\|
\Lindbladian(\prmtr)\rho^{(r)}(\prmtr)
\bigr\| ^2
\; = \;
x(\prmtr)^\top
A_r(\prmtr) \,
x(\prmtr).
\end{equation*}
The symmetric positive semidefinite matrix $A_r(\prmtr)$ is
the Galerkin projection of $\mathcal A(\prmtr)$ to~$V$ and is used in both
the greedy and online phases.

Next, we describe how $A_r(\prmtr)$ can be assembled efficiently.
Recall the
affine decomposition \eqref{eq:Lindbladian_affine_decomposition} 
from  \Cref{ass:affine_decomp_param_Lindbladian}. 
For the offline assembly, we precompute the action of every component operator
on the reduced basis,
i.e., $\mathcal K_\IdxHam B_j$ and $\mathcal D_\IdxDis B_j$ 
for $\IdxHam=1,\ldots,\NrHam$, 
$\IdxDis=1,\ldots,\NrDis$
and $j=1,\ldots,r$.
Using these quantities, define the parameter-independent matrices
$A_{r,\mathrm{HH}}^{\IdxHam,\IdxHam'}, 
A_{r,\mathrm{HD}}^{\IdxHam,\IdxDis}, 
A_{r,\mathrm{DD}}^{\IdxDis,\IdxDis'}\in\R^{r\times r}$ 
by
\begin{align}
\bigl(A_{r,\mathrm{HH}}^{\IdxHam,\IdxHam'}\bigr)_{ij}
&\coloneqq
\bigl\langle
\mathcal K_\IdxHam B_j,
\mathcal K_{\IdxHam'} B_i
\bigr\rangle ,
\label{eq:offline_HH_matrix}
\\
\bigl(A_{r,\mathrm{HD}}^{\IdxHam,\IdxDis}\bigr)_{ij}
&\coloneqq
\bigl\langle
\mathcal K_\IdxHam B_j,
\mathcal D_\IdxDis B_i
\bigr\rangle 
+
\bigl\langle
\mathcal D_\IdxDis B_j,
\mathcal K_\IdxHam B_i
\bigr\rangle ,
\label{eq:offline_HD_matrix}
\\
\bigl(A_{r,\mathrm{DD}}^{\IdxDis,\IdxDis'}\bigr)_{ij}
&\coloneqq
\bigl\langle
\mathcal D_\IdxDis B_j,
\mathcal D_{\IdxDis'} B_i
\bigr\rangle .
\label{eq:offline_DD_matrix}
\end{align}
Here,
$\IdxHam,\IdxHam'= 1,\ldots,\NrHam $,
$\IdxDis,\IdxDis'= 1,\ldots,\NrDis $
and $i,j = 1,\ldots, r$.

Expanding \eqref{eq:reduced_normal_matrix} with the affine decomposition of
$\Lindbladian(\prmtr)$ yields
\begin{align}
A_r(\prmtr)
\;=\;
\sum_{\IdxHam,\IdxHam'=1}^{\NrHam}
\theta_\IdxHam(\prmtr)\theta_{\IdxHam'}(\prmtr)
A_{r,\mathrm{HH}}^{\IdxHam,\IdxHam'}
\,+\,
\sum_{\IdxHam=1}^{\NrHam}
\sum_{\IdxDis=1}^{\NrDis}
\theta_\IdxHam(\prmtr)\phi_\IdxDis(\prmtr)
A_{r,\mathrm{HD}}^{\IdxHam,\IdxDis}
\,+
\sum_{\IdxDis,\IdxDis'=1}^{\NrDis}
\phi_\IdxDis(\prmtr)\phi_{\IdxDis'}(\prmtr)
A_{r,\mathrm{DD}}^{\IdxDis,\IdxDis'}.
\label{eq:online_assembly_reduced_normal_matrix}
\end{align}
The matrices in
\eqref{eq:offline_HH_matrix}--\eqref{eq:offline_DD_matrix} are computed offline.
Using their symmetry, 
i.e., 
$(A_{r,\mathrm{HH}}^{\IdxHam,\IdxHam'})^\top
=
A_{r,\mathrm{HH}}^{\IdxHam',\IdxHam}$, 
$(A_{r,\mathrm{DD}}^{\IdxDis,\IdxDis'})^\top
=
A_{r,\mathrm{DD}}^{\IdxDis',\IdxDis}$ and
$A_{r,\mathrm{HD}}^{\IdxHam,\IdxDis}$ being symmetric,
the number of stored $r\times r$ matrices
reduces to ${(\NrHam+\NrDis)(\NrHam+\NrDis+1)}/2$.
For a new parameter $\prmtr$, the reduced matrix $A_r(\prmtr)$ is computed
using only evaluations of the scalar coefficient functions and operations on $r\times r$ matrices.
In particular, no application of the full-order Lindbladian is required.

The offline quantities can also be updated incrementally during the greedy construction: 
when a new orthonormal basis matrix $B_{r+1}$ is added, 
only the actions $\mathcal K_\IdxHam B_{r+1}$ and $\mathcal D_\IdxDis B_{r+1} $
and the corresponding new rows and columns of the reduced matrices need to be computed; 
all previously assembled entries remain unchanged.

\subsection{Reduced formulations and greedy selection rules}
\label{subsec:RBM_formulations}

We introduce two reduced formulations based on the matrix
$A_r(\prmtr)$ from
\Cref{subsec:reduced_operator_assembly}.
These formulations are relevant for both the greedy selection in the offline phase
and the reduced steady state approximation in the online phase. 

Let us also introduce the parameter-independent vector
\begin{equation*}
t_r
\coloneqq
(
\tr B_1,\ldots,\tr B_r
)^\top
\in\R^{r}, 
\end{equation*}
which is used to represent the trace of an element in $V$.

\subsubsection{Reduced eigenvalue formulation} 
\label{subsubsec:EVP_formulation}
Following \eqref{eq:abstract_EVP}, 
we consider $x_{\mathrm n}^{(\rdim)}(\prmtr)\in\R^\rdim$ with
\begin{equation*}
x_{\mathrm n}^{(r)}(\prmtr)
\, \in
\underset{
x\in\R^\rdim,\ \|x\|_2=1
}{
\operatorname{argmin}
}
x^\top A_\rdim(\prmtr) \, x,
\end{equation*}
where $\|\cdot\|_2$ denotes the canonical Euclidean norm of $\mathbb{R}^\rdim$,
and we define
\begin{equation}
\rho_{\mathrm n}^{(\rdim)}(\prmtr)
\, \coloneqq
\sum_{j = 1 }^\rdim
\bigl(x_{\mathrm n}^{(r)}(\prmtr)\bigr)_j B_j.
\label{eq:norm_normalized_reduced_eigenvector}
\end{equation}
Equivalently, $x_{\mathrm n}^{(\rdim)}(\prmtr)$ is a normalized eigenvector
associated with the smallest eigenvalue of $A_r(\prmtr)$.

For the \emph{greedy construction} of the reduced space in the eigenvalue formulation, 
we use the squared residual norm 
as the error indicator, i.e., 
\begin{align}
\Delta_{\mathrm{eig}}^{(r)}(\prmtr)
\,\coloneqq \, 
\|\Lindbladian(\prmtr)\rho_{\mathrm n}^{(r)}(\prmtr)\| ^2
\, = \, 
{
\bigl(x_{\mathrm n}^{(r)}(\prmtr)\bigr)^\top
A_r(\prmtr) \, 
x_{\mathrm n}^{(r)}(\prmtr)
}
\, = \,
\lambda_{\min}\bigl(A_r(\prmtr)\bigr).
\label{eq:eigenvalue_error_indicator}
\end{align}
The greedy algorithm using this error estimator is called \emph{EigGreedy}.

If $\tr \rho_{\mathrm n}^{(r)}(\prmtr)
=
t_r^\top x_{\mathrm n}^{(r)}(\prmtr)
\neq 0,$
the unit-trace reduced approximation is set to be 
\begin{equation}
\rho_{\mathrm{eig}}^{(r)}(\prmtr)
\, \coloneqq
\frac{
\rho_{\mathrm n}^{(r)}(\prmtr)
}{
\tr \rho_{\mathrm n}^{(r)}(\prmtr)
}.
\label{eq:trace_normalized_reduced_eigenvector}
\end{equation}
Equivalently, its coefficient vector is given by ${
x_{\mathrm n}^{(r)}(\prmtr)
}/{
(t_r^\top x_{\mathrm n}^{(r)}(\prmtr))
}$. 
Note that this trace normalization is performed only in the \emph{online} phase.

\subsubsection{Affine residual formulation}
\label{subsubsec:affine_res_formulation}

Motivated by \eqref{eq:abstract_constraint_res_min},
the affine residual formulation incorporates the unit-trace condition directly
into the reduced solve. 
Its coefficient vector is defined by
\begin{equation}
x_{\mathrm{aff}}^{(r)}(\prmtr)
\in
\underset{
x\in\R^r,\ t_r^\top x=1
}{
\operatorname{argmin}
}
\frac{1}{2}x^\top A_r(\prmtr)x.
\label{eq:reduced_affine_residual_problem}
\end{equation}
For the greedy spaces, feasibility follows from $I/N \in V=V_r$.
For a nontrivial POD space, the leading POD mode can be chosen to have
nonzero trace, since the snapshot Gram matrix has nonnegative entries.
Hence, in either construction, $V\cap A_1\neq\emptyset$.
The first-order optimality conditions of \eqref{eq:reduced_affine_residual_problem} 
yield the saddle-point system
\begin{equation}
\begin{pmatrix}
A_r(\prmtr) & t_r\\
t_r^\top & 0
\end{pmatrix}
\begin{pmatrix}
x_{\mathrm{aff}}^{(r)}(\prmtr)\\
\eta^{(r)}(\prmtr)
\end{pmatrix}
=
\begin{pmatrix}
0\\
1
\end{pmatrix},
\label{eq:reduced_affine_KKT_system}
\end{equation}
where $\eta^{(r)}(\prmtr)\in\R$ denotes the Lagrange multiplier.
If the minimizer in \eqref{eq:reduced_affine_residual_problem} is not unique,
\eqref{eq:reduced_affine_KKT_system} is understood as selecting any solution of
the solvable saddle-point system.  
If the full problem has a unique steady
state, $\mathcal A(\prmtr)$ is positive definite on~$A_0$.
Hence,  
the reduced objective is strictly convex 
on the feasible affine space and the saddle-point
system is nonsingular.

The resulting reduced approximation is then given by 
\begin{equation}
\rho_{\mathrm{aff}}^{(r)}(\prmtr)
\coloneqq
\sum_{j = 1 }^r
\bigl(x_{\mathrm{aff}}^{(r)}(\prmtr)\bigr)_j B_j.
\label{eq:reduced_affine_approximation}
\end{equation}
By construction, it holds that $\tr \rho_{\mathrm{aff}}^{(r)}(\prmtr)
=
t_r^\top x_{\mathrm{aff}}^{(r)}(\prmtr)
=
1.
$
Moreover,
by writing $ X(\prmtr) \coloneqq \rho_{\mathrm{aff}}^{(r)}(\prmtr) - I/N \in V \cap A_0$,
we have 
$$
\Lindbladian(\prmtr)\rho_{\mathrm{aff}}^{(r)}(\prmtr)  
\, = \,
\Lindbladian(\prmtr)X(\prmtr) 
+ N^{-1} \Lindbladian(\prmtr)I. 
$$
In other words, 
the affine residual formulation
corresponds to approximating the steady state problem in $V \cap A_0$ from the linear system point of view, 
cf. \Cref{rem:steady_state_eqn_as_LES}.

For the greedy construction, 
we again use 
the squared residual norm 
as the error indicator, i.e.,
\begin{align}
\Delta_{\mathrm{aff}}^{(r)}(\prmtr)
\, \coloneqq \,
\|\Lindbladian(\prmtr)\rho_{\mathrm{aff}}^{(r)}(\prmtr)\| ^2
\, = \,
{
\bigl(x_{\mathrm{aff}}^{(r)}(\prmtr)\bigr)^\top
A_r(\prmtr) \, 
x_{\mathrm{aff}}^{(r)}(\prmtr)
}
.
\label{eq:affine_error_indicator}
\end{align}
The greedy algorithm using this error estimator is called \emph{AffGreedy}.

\subsubsection{Algorithmic formulation and computational comparison}

For either formulation
$\star\in\{\mathrm{eig},\mathrm{aff}\}$,
the next snapshot parameter $\prmtr_{\mathrm{next},\star}$ is selected as
the maximizer of the corresponding indicator $\Delta_{\star}^{(r)}$ over the training set $\Xi_{\mathrm{train}}$.
The steady state $\SteadyState(\prmtr_{\mathrm{next},\star})$ is then computed,
orthonormalized against the current reduced basis, and used to enrich the
reduced space.

The construction starts from $V_1=\sspan_\R\{I\}$, motivated
by $A_1=I/N+A_0$.  Since the nested reduced spaces give
$\Delta_\star^{(r+1)}(\prmtr)\leq \Delta_\star^{(r)}(\prmtr)$, parameters
whose indicators are below tolerance can be removed from subsequent sweeps.
In other words,
we only need to evaluate $\Delta_\star^{(r)}$ on an \emph{active training set} $\Xi_{\mathrm{act}} \subseteq \Xi_{\mathrm{train}}$. 
The resulting procedure is summarized in \Cref{alg:greedy_reduced_basis}. 

At every active parameter, the eigenvalue formulation
computes the smallest eigenpair of $A_r(\prmtr)$, whereas the affine
formulation solves the $(r+1)$-dimensional saddle-point system
\eqref{eq:reduced_affine_KKT_system}.  Both dense solves have cubic worst-case
complexity in~$r$; their practical costs are compared in \Cref{sec:num}.

\begin{algorithm}[htbp]
\caption{Greedy construction of the reduced space}
\label{alg:greedy_reduced_basis}
\begin{algorithmic}[1]

\REQUIRE
Data of the parametric Lindbladian
$\Lindbladian(\prmtr)$ satisfying
\Cref{ass:affine_decomp_param_Lindbladian},
training set $\Xi_{\mathrm{train}}\subseteq\prmtrSet$,
indicator tolerance $\varepsilon_{\mathrm{tol}}>0$,
maximum reduced dimension $J_{\max}\geq1$,
and a formulation $\star\in\{\mathrm{eig},\mathrm{aff}\}$

\ENSURE
Reduced space $V_r$ with Frobenius-orthonormal basis
$B_1,\ldots,B_r$

\STATE Set
$V_1\gets\sspan_{\R}\{I\}$,
$B_1\gets I/\sqrt N$,
and $r\gets 1 $

\STATE Set
$\Xi_{\mathrm{act}}\gets\Xi_{\mathrm{train}}$

\WHILE{$r<J_{\max}$}

  \FOR{all $\prmtr\in\Xi_{\mathrm{act}}$}
    \STATE Assemble $A_r(\prmtr)$ using
    \eqref{eq:online_assembly_reduced_normal_matrix}
    \STATE Solve the reduced problem corresponding to
    $\star\in\{\mathrm{eig},\mathrm{aff}\}$
    \STATE Compute
    $\Delta_{\star}^{(r)}(\prmtr)$ using
    \eqref{eq:eigenvalue_error_indicator} or
    \eqref{eq:affine_error_indicator}
  \ENDFOR

  \STATE Set $  \Xi_{\mathrm{act}}
  \gets
  \{
  \prmtr\in\Xi_{\mathrm{act}}:
  \Delta_{\star}^{(r)}(\prmtr)
  >
  \varepsilon_{\mathrm{tol}}
  \}$

  \IF{$\Xi_{\mathrm{act}}=\emptyset$}
    \STATE \textbf{break}
  \ENDIF

  \STATE Select
  $ \prmtr_{\mathrm{next},\star} 
  \in
  \operatorname*{arg\,max}_{\prmtr\in\Xi_{\mathrm{act}}}
  \Delta_{\star}^{(r)}(\prmtr) $

  \STATE Compute $\SteadyState(\prmtr_{\mathrm{next},\star})$ 
  and orthonormalize it against $B_1,\dots, B_r$ to obtain $B_{r+1}$

  \STATE Set $  V_{r+1}
  \gets
  V_r+
  \sspan_{\R}
  \{\SteadyState(\prmtr_{\mathrm{next},\star})\}$

  \STATE Set $r\gets r+1$

\ENDWHILE

\RETURN $V_r$ and $B_1,\ldots,B_r$

\end{algorithmic}
\end{algorithm}

\subsection{Positivity post-processing}
\label{subsec:post_processing}

For an online approximation
$\rho_\star^{(r)}(\prmtr)\in A_1$, we consider two optional positivity
corrections.

\subsubsection{Projection onto the quantum-state set}
\label{subsubsec:metric_proj}

For an arbitrary $\rho \in A_1$ with spectral decomposition
\begin{equation}
\rho
=
U \operatorname{diag}(\lambda_1,\ldots,\lambda_{\fdim}) U^\dagger
\quad
\text{for} \quad
\lambda_1\geq\cdots\geq\lambda_{\fdim},
\label{eq:projection_spectral_decomposition}
\end{equation}
the Frobenius metric projection of $\rho$ onto $S$ is defined as
$\Pi_S(\rho )
\coloneqq
\operatorname*{arg\,min}_{Y\in S}
\|
Y-\rho 
\| .$
Since $S$ is closed and convex, this projection exists and is unique.

Moreover, $\Pi_S(\rho)$
has the same eigenbasis as~$\rho$: 
for any $Y\in S$, observe $Z = (Z_{ij})_{i,j=1}^\fdim \coloneqq U^\dagger YU \in S$ by unitarity of $U$
and 
$Z_{\mathrm{D}} := \operatorname{diag}(Z_{11},\ldots,Z_{\fdim\fdim}) \in S$ due to $\tr Z_{\mathrm{D}}= \tr Z = 1$ and $e_j^\top Z_{\mathrm{D}} e_j \geq 0$ for $e_j$ the $j$-th standard basis vector. 
Hence, $\widehat Y := U Z_{\mathrm{D}} U^\dagger \in S$.
Denoting $\Lambda \coloneqq \operatorname{diag}(\lambda_1,\ldots,\lambda_{\fdim})$,
we deduce with unitary invariance of $\|\cdot\|$
and the definitions of $Z_\mathrm{D}, Z$ that 
$$
\|\widehat Y-\rho\|^2
= \|Z_\mathrm{D} -\Lambda\|^2
\leq \|Z-\Lambda\|^2
\leq\|Y-\rho\|^2.$$ 
Equality is attained for $Y$ diagonal in the eigenbasis of $\rho$;
see also \cite[Sec.~II]{SmolinGambettaSmith2012}. 

Hence, computing $\Pi_S(\rho )$ 
is reduced to projecting the eigenvalue vector 
$(\lambda_1,\ldots,\lambda_\fdim)^\top$ onto the simplex 
$\{y=(y_i) \in \R^\fdim : y_i \geq 0, \sum_i y_i = 1 \}$. 
To this end, we follow 
\cite{WangCarreiraPerpinan2013} and set
\begin{equation}
k_*
\,\coloneqq
\max
\Bigl\{
k\in\{1,\ldots,\fdim\}:
\lambda_k
+
\frac{1}{k}
\Bigl(
1-\sum_{j=1}^k\lambda_j
\Bigr)
>0
\Bigr\}
\quad \text{and} \quad
s
\coloneqq
\frac{1}{k_*}
\Bigl(
1-\sum_{j=1}^{k_*}\lambda_j
\Bigr)
.
\label{eq:projection_active_index_projection_shift}
\end{equation}
The projected eigenvalues are
\begin{equation}
\lambda_j^{\mathrm{proj}}
\coloneqq
\max\{\lambda_j+s,0\}
\qquad \text{for }
j=1,\ldots,\fdim,
\label{eq:projection_eigenvalues}
\end{equation}
and we obtain
\begin{equation}
\Pi_S(\rho)
=
U
\operatorname{diag}
\bigl(
\lambda_1^{\mathrm{proj}},\ldots,
\lambda_{\fdim}^{\mathrm{proj}}
\bigr)
U^\dagger.
\label{eq:projection_reconstruction}
\end{equation}

The projection algorithm is summarized in \Cref{alg:projection_onto_state_set}.
It yields the unique best approximation in the Frobenius norm,
but requires a complete spectral decomposition of an
$\fdim\times\fdim$ Hermitian matrix.
Hence, the algorithm may become prohibitively expensive for large \(\fdim\).

\begin{algorithm}[htbp]
\caption{Projection onto the quantum-state set}
\label{alg:projection_onto_state_set}
\begin{algorithmic}[1]
\REQUIRE Hermitian matrix $\rho\in A_1$
\ENSURE Metric projection $\rho_{\mathrm{proj}}=\Pi_S(\rho)\in S$
\STATE Compute the spectral decomposition \eqref{eq:projection_spectral_decomposition}.
\STATE  Compute $k_*$ and $s$ from
\eqref{eq:projection_active_index_projection_shift}. 
\STATE Compute the projected eigenvalues from
\eqref{eq:projection_eigenvalues}.
\STATE Reconstruct $\rho_{\mathrm{proj}}$ using
\eqref{eq:projection_reconstruction}.
\RETURN $\rho_{\mathrm{proj}}$
\end{algorithmic}
\end{algorithm}

\subsubsection{Partial spectral clipping}

To avoid a complete spectral decomposition, one may compute only the
$k$ algebraically smallest eigenpairs
\begin{equation}
(\lambda_\fdim,u_\fdim),\ldots,(\lambda_{\fdim-k+1},u_{\fdim-k+1})
\quad \text{for} \quad 
\lambda_\fdim\leq\cdots\leq \lambda_{\fdim-k+1} 
\label{eq:partial_clipping_eigenpairs}
\end{equation}
of a matrix $\rho\in A_1$, 
with $k \ll \fdim$.  
Then, compute
\begin{equation}
C_k(\rho)
\coloneqq
\!\!\!
\sum_{j=\fdim-k+1}^\fdim
\!\!\!
\max\{-\lambda_j,0\}\,
u_j u_j^\dagger
\quad \text{and} \quad 
s_k(\rho)
\coloneqq
\tr C_k(\rho)
=
\!\!\!
\sum_{j=\fdim-k+1}^\fdim
\!\!\!
\max\{-\lambda_j,0\}.
\label{eq:partial_clipping_corrections}
\end{equation}
The partially clipped approximation is then given by
\begin{equation}
\Pi_{\mathrm{clip},k}(\rho)
\coloneqq
\frac{\rho+C_k(\rho)}{1+s_k(\rho)}.
\label{eq:partial_clipping_map}
\end{equation}
The output is Hermitian and has unit trace. 
If all negative eigenvalues of
$\rho$ are included among the computed eigenpairs, then
$\Pi_{\mathrm{clip},k}(\rho)\in S$. In particular, this is certified when the
$k$ algebraically smallest eigenvalues are computed and $\lambda_{\fdim-k+1}\geq0$.
Otherwise, only the detected negative spectral components are corrected.
The procedure is summarized in \Cref{alg:partial_spectral_clipping}. 

\begin{algorithm}[htbp]
\caption{Partial spectral clipping}
\label{alg:partial_spectral_clipping}
\begin{algorithmic}[1]
\REQUIRE Hermitian matrix $\rho\in A_1$ and
$k\in\{1,\ldots,\fdim\}$ with $k\ll \fdim$
\ENSURE Partially corrected matrix
$\rho_{\mathrm{clip},k}\in A_1$
\STATE Compute the eigenpairs in
\eqref{eq:partial_clipping_eigenpairs}.
\STATE Compute $C_k(\rho)$ and $s_k(\rho)$ from
\eqref{eq:partial_clipping_corrections}.
\STATE Set
$\rho_{\mathrm{clip},k}\coloneqq\Pi_{\mathrm{clip},k}(\rho)$
using \eqref{eq:partial_clipping_map}.
\RETURN $\rho_{\mathrm{clip},k}$
\end{algorithmic}
\end{algorithm}

The exact projection requires a complete diagonalization, whereas partial
spectral clipping can be implemented using an iterative extremal-eigenvalue
solver and may be substantially cheaper when $k\ll\fdim$. 
Since both procedures require full-matrix reconstruction and spectral computations, 
we propose to only provide them optionally in the online phase.

\section{Error analysis and greedy convergence}
\label{sec:error-analysis}

We now analyze the approximation errors 
and greedy procedures introduced in the previous section. 
\Cref{subsec:postprocessing_effect} discusses the effect of positivity postprocessing.
In \Cref{subsec:state-error}, 
we relate the pointwise Lindbladian residual 
to the distance from its kernel and, 
under a pointwise non-degeneracy assumption, 
to the error with respect to the unique steady state.
We then recall in \Cref{subsec:weak-greedy} the weak greedy convergence framework 
for RBM and verify the validity of the two greedy selection rules.
Finally, in \Cref{subsec:width-convergence}, we combine analytic regularity of
the steady-state map with polynomial interpolation and width estimates to
derive (sub-)exponential convergence for the idealized full-domain greedy procedures,
and then relate this result to the finite training set used in \Cref{alg:greedy_reduced_basis}.

\subsection{Effect of positivity post-processing}\label{subsec:postprocessing_effect}

Let us quantify the additional approximation error introduced by the two
post-processing procedures from \Cref{subsec:post_processing}. 

\begin{lemma}[Error after projection onto the state set]
\label[lemma]{prop:projection_postprocessing_error}
For $\rho\in A_1$ and
$\SteadyState(\prmtr)\in S$,
it holds that 
\begin{equation*}
\left\|
\Pi_S(\rho)
-
\SteadyState(\prmtr)
\right\| 
\leq
\left\|
\rho-\SteadyState(\prmtr)
\right\| .
\end{equation*}
\end{lemma}

\begin{proof}
Metric projections onto a nonempty closed convex subset 
of a Hilbert space are nonexpansive 
by \cite[Prp.~4.8]{BauschkeCombettes2017}.
Hence, for $\Pi_S$ from \Cref{subsubsec:metric_proj} and $S$ convex,
we have $\|\Pi_S(X)-\Pi_S(Y)\|\leq\|X-Y\|$ for all $X,Y\in\mathcal{H}$. 
With $\Pi_S\bigl(\SteadyState(\prmtr)\bigr) = \SteadyState(\prmtr) \in S$,
the assertion follows.
\end{proof}

Thus, exact projection onto $S$ does not increase the state approximation
error.

\begin{lemma}[Error after partial spectral clipping]
\label[lemma]{prop:partial_clipping_error}
Let $\rho\in A_1$, $\SteadyState(\prmtr)\in S$, 
and $\Pi_{\mathrm{clip},k}(\rho)$ the partially clipped approximation from
\eqref{eq:partial_clipping_map}. 
Then, it holds that 
\begin{equation*}
\left\|
\Pi_{\mathrm{clip},k}(\rho)
-
\SteadyState(\prmtr)
\right\| 
\, \leq \, 
\bigl(2+\sqrt{n_{\mathrm{clip}}}\bigr)
\left\|
\rho-\SteadyState(\prmtr)
\right\| ,
\end{equation*}
where $n_{\mathrm{clip}} $
 is the number of actually clipped eigenvalues, 
i.e.,
\begin{equation*}
n_{\mathrm{clip}} 
\coloneqq
\rank C_k(\rho)
=
\operatorname{card}
\bigl\{
j\in\{\fdim-k+1,\ldots,\fdim\}:
\lambda_j<0
\bigr\} 
\leq k.
\end{equation*}
\end{lemma}

\begin{proof}
For brevity, set
$C_k\coloneqq C_k(\rho)$,
$s_k\coloneqq s_k(\rho)$ 
and
$E\coloneqq\rho-\SteadyState(\prmtr)$.
Let $\mu_1\geq\cdots\geq\mu_\fdim\geq0$ denote the eigenvalues of
$\SteadyState(\prmtr)$. Since
$\max\{-\lambda_j,0\}\leq|\lambda_j-\mu_j|$,
the Hoffman--Wielandt inequality
\cite[Thm.~6.3.5 and Cor.~6.3.8]{HornJohnson2012} yields
\[
\begin{aligned}
\|C_k\|^2
=
\sum_{j=\fdim-k+1}^{\fdim}
\max\{-\lambda_j,0\}^2 
\leq
\sum_{j=1}^{\fdim}|\lambda_j-\mu_j|^2
\leq
\|E\|^2.
\end{aligned}
\]
Moreover, $C_k$ is positive semidefinite and has rank
$n_{\mathrm{clip}}$. 
Hence, we have 
\[
s_k
=
\tr C_k
\leq
\sqrt{n_{\mathrm{clip}}}\,\|C_k\|
\leq
\sqrt{n_{\mathrm{clip}}}\,\|E\|.
\]
Finally, using \eqref{eq:partial_clipping_map},
$\|\SteadyState(\prmtr)\|\leq1$ and $1+s_k\geq1$, we obtain
\[
\begin{aligned}
\left\|
\Pi_{\mathrm{clip},k}(\rho)-\SteadyState(\prmtr)
\right\|
=
\frac{
\left\|E+C_k-s_k\SteadyState(\prmtr)\right\|
}{1+s_k} 
\leq
\|E\|+\|C_k\|+s_k\|\SteadyState(\prmtr)\| 
\leq
\bigl(2+\sqrt{n_{\mathrm{clip}}}\bigr)\|E\|,
\end{aligned}
\]
which proves the assertion.
\end{proof}

To summarize, both positivity post-processing procedures 
only scale the steady state approximation error by a prefactor, 
which is $1$ for the exact projection onto the quantum state set $S$, 
and at most $2+\sqrt{k}$ for partial clipping.

\subsection{Residual-based state-error estimates}
\label{subsec:state-error}

In this subsection, 
let us fix a reduced space $V \subseteq\mathcal H$,
constructed by POD or greedy procedures.
We denote the corresponding reduced approximation by
$\rho_{\star}^{(r)}(\prmtr)$ for
$\star\in\{\mathrm{eig},\mathrm{aff}\}$ obtained by 
\eqref{eq:trace_normalized_reduced_eigenvector} or \eqref{eq:reduced_affine_approximation}, respectively. 
The goal of this subsection is to relate the RB approximation error to the residual.

For $\prmtr\in\prmtrSet$, recall
$\mathcal A(\prmtr)=\Lindbladian(\prmtr)^\dagger\Lindbladian(\prmtr)$ and
$\ker\Lindbladian(\prmtr)=\ker\mathcal A(\prmtr)$.  
We introduce 
\begin{align}
\gamma(\prmtr)
\coloneqq
\min_{\substack{X\perp\ker\Lindbladian(\prmtr)\\ \|X\|=1}}
\langle X,\mathcal A(\prmtr)X\rangle
\qquad \text{and} \qquad
\beta(\prmtr)
\coloneqq
\min_{\substack{X\in A_0\\ \|X\|=1}}
\|\Lindbladian(\prmtr)X\|.
\label{eq:gap_and_trace_zero_stability}
\end{align}
Thus, $\gamma(\prmtr)$ is the smallest positive eigenvalue of
$\mathcal A(\prmtr)$,
which represents the spectral gap of $\mathcal A(\prmtr)$
and controls the distance to the kernel, 
whereas $\beta(\prmtr)$ is the $\inf$-$\sup$-constant 
on the trace-zero space and 
directly controls trace-zero errors.

By \Cref{lem:Lindbladian_properties},
$\ker\Lindbladian(\prmtr)$ contains a steady state. 
 Hence, the restriction of
the trace functional to $\ker\Lindbladian(\prmtr)$ is surjective and has kernel
$A_0\cap\ker\Lindbladian(\prmtr)$.  
Applying the rank--nullity theorem to this restricted trace functional 
yields $ 1 = \dim \ker \Lindbladian(\prmtr) - \dim (A_0 \cap \ker\Lindbladian(\prmtr)) $.
Consequently, we obtain
\begin{equation}
\beta(\prmtr)>0
\quad\Longleftrightarrow\quad
A_0\cap\ker\Lindbladian(\prmtr)=\{0\}
\quad\Longleftrightarrow\quad
\Lindbladian(\prmtr)\text{ has a unique steady state}.
\label{eq:equivalence_unique_steady_state}
\end{equation}

To relate $\beta(\prmtr)$ to $\gamma(\prmtr)$, 
we introduce the auxiliary minimal angle $\psi(\prmtr)\in[0,\pi/2]$ 
between $A_0$ and $\ker\Lindbladian(\prmtr)$ by
\begin{equation*}
\sin\psi(\prmtr)
\coloneqq
\inf_{\substack{X\in A_0\\\|X\|=1}}
\operatorname{dist}\bigl(X,\ker\Lindbladian(\prmtr)\bigr).
\end{equation*}

\begin{lemma}[Gap--angle bound for trace-zero stability]
\label[lemma]{lem:beta_gap_angle_relation}
For every $\prmtr\in\prmtrSet$, it holds that
\begin{equation}
    \beta(\prmtr)
\geq
\sqrt{\gamma(\prmtr)}\sin\psi(\prmtr).
\label{eq:beta_gamma_angle_rel}
\end{equation}
If $\Lindbladian(\prmtr)$ has a unique steady state, 
then it holds that 
\begin{equation*}
\sin\psi(\prmtr)
=
\frac{1}{\sqrt{\fdim}\,\|\SteadyState(\prmtr)\|}
\in
\left[\frac{1}{\sqrt{\fdim}},1\right]
\qquad \text{and} \qquad
\beta(\prmtr)
\geq
\frac{\sqrt{\gamma(\prmtr)}}
{\sqrt{\fdim}\,\|\SteadyState(\prmtr)\|}
\geq
\sqrt{\frac{\gamma(\prmtr)}{\fdim}}.
\end{equation*}
\end{lemma}

\begin{proof}
For $X\in A_0$, decompose $X = X_{\parallel} + X_{\perp}$ 
with $X_{\parallel} \in \ker\Lindbladian(\prmtr)$ and $X_{\perp} \in \ker\Lindbladian(\prmtr)^\perp$. 
Note that $\|\Lindbladian(\prmtr)X\|^2 = \langle X_\perp, \mathcal{A}(\prmtr)X_\perp\rangle $.
Hence, the definition of $\gamma(\prmtr)$ and the minimal angle $\psi(\prmtr)$ in
\eqref{eq:gap_and_trace_zero_stability} give
\[
\|\Lindbladian(\prmtr)X\|^2
\geq
\gamma(\prmtr)
\operatorname{dist}\bigl(X,\ker\Lindbladian(\prmtr)\bigr)^2
\geq
\gamma(\prmtr)\, \sin^2\!\psi(\prmtr)\, \|X\|^2.
\]
Taking the minimum over $X\in A_0$ with $\|X\|=1$ proves \eqref{eq:beta_gamma_angle_rel}.

In the unique steady-state case, 
we set $u \coloneqq \SteadyState(\prmtr) / \|\SteadyState(\prmtr)\|$
and observe that 
$\ker\Lindbladian(\prmtr)  
= \sspan\{\SteadyState(\prmtr)\}
= \sspan \{u\}$.
Using this fact and $\|u\|=1$,
we deduce that 
$$
\sin^2\!\psi(\prmtr)
=
\inf_{\substack{X\in A_0\\\|X\|=1}}
\dist(X, \ker \Lindbladian(\prmtr))^2
=
\inf_{\substack{X\in A_0\\\|X\|=1}}
1- |\langle X,u\rangle|^2
=
1- \sup_{\substack{X\in A_0\\\|X\|=1}}
|\langle X,u\rangle|^2.
$$ 
Moreover, 
observe that 
$\sup_{\substack{X\in A_0,\|X\|=1}}
|\langle X,u\rangle|^2
= \|P_{A_0}u\|^2
$,
where $P_{A_0}$ is the orthogonal projector onto $A_0$.
Hence, by setting 
$n \coloneqq I / \sqrt{\fdim}$,
which satisfies $n \perp A_0$, 
we obtain 
$1-\|P_{A_0}u\|^2 = |\langle u, n \rangle|^2$.
Combining these considerations and using $\tr\SteadyState(\prmtr)=1$,
we see that 
\[
\begin{aligned}
\sin^2\!\psi(\prmtr)
=
1- \sup_{\substack{X\in A_0\\\|X\|=1}}
|\langle X,u\rangle|^2
=
1-\|P_{A_0}u\|^2  
=
|\langle u,n\rangle|^2
=
\frac{1}{\fdim\|\SteadyState(\prmtr)\|^2}.
\end{aligned}
\]
The rest follows from $\|\SteadyState(\prmtr)\|\in [1/\sqrt{\fdim},1]$;
see \Cref{lem:geometric_properties}.
\end{proof}

Thus, the gap and the angle provide 
a possibly coarse lower bound 
for $\beta(\prmtr)$,
and we will primarily use $\beta(\prmtr)$ in the following error estimates.

\begin{proposition}[Residual-based state-error estimate]
\label[proposition]{prop:generic_residual_state_error}
Let $\prmtr\in \prmtrSet$ satisfy $\dim \ker\Lindbladian(\prmtr) = 1$, 
and let $X\in\mathcal H$ satisfy $\tr X\neq0$. 
Then, it holds that 
\begin{equation*}
\Bigl\|
\SteadyState(\prmtr)
-
\frac{X}{\tr X}
\Bigr\|
\,\leq\,
\frac{
\|\Lindbladian(\prmtr)X\|
}{
\beta(\prmtr)|\tr X|
}.
\end{equation*}
\end{proposition}

\begin{proof}
Set $E
\coloneqq
\SteadyState(\prmtr)
-{X}/{\tr X}
.$
Since both terms have unit trace, $E\in A_0$.  
The definition of
$\beta(\prmtr)$ and $\Lindbladian(\prmtr)\SteadyState(\prmtr)=0$ 
give
$\beta(\prmtr)\|E\|
\leq
\|\Lindbladian(\prmtr)E\|
=
{\|\Lindbladian(\prmtr)X\|}/{|\tr X|},$
which proves the claim.
\end{proof}

Finally, 
we concretize \Cref{prop:generic_residual_state_error}
in the context of steady state RB approximation. 

\begin{corollary}[State error for the reduced formulations]
\label[corollary]{cor:online_approx_steady_state_error}
Let $\prmtr\in \prmtrSet$ satisfy $\dim \ker \Lindbladian(\prmtr) =1$.
Let $V \subseteq \mathcal{H}$ be an RB space constructed by one of the methods in \Cref{subsubsec:POD_vs_greedy}.
Consider $ \rho_{\mathrm{aff}}^{(r)}(\prmtr) $
from \eqref{eq:reduced_affine_approximation},
and assume that $ \rho_{\mathrm{eig}}^{(r)}(\prmtr)$ from \eqref{eq:trace_normalized_reduced_eigenvector} is well-defined. 
For $\star \in \{ \mathrm{eig}, \mathrm{aff}\}$, 
it holds that 
\begin{equation*}
\|
\SteadyState(\prmtr)-\rho_{\star}^{(r)}(\prmtr)
\|
\,\leq\,
\frac{
\|\Lindbladian(\prmtr)\rho_{\star}^{(r)}(\prmtr)\|
}{
\beta(\prmtr)
}.
\end{equation*}
\end{corollary}

\subsection{Weak greedy convergence framework for RBM and validity of the selection rules}
\label{subsec:weak-greedy}

We would like to establish a priori convergence results 
for the greedy construction under some further conditions that will be specified.
We first consider an idealized version in which the maximization is performed over the full
parameter set~$\prmtrSet$, and return to the finite training set from
\Cref{alg:greedy_reduced_basis} in \Cref{subsec:width-convergence}.

\subsubsection{Generic weak greedy convergence framework for RBM}

To begin with, let us recall the weak greedy convergence result from \cite{DeVorePetrovaWojtaszczyk2013,BinevEtAl2011,BuffaEtAl2012}.

\begin{theorem}[Weak greedy convergence rates for RBM {\cite[Cor.~3.3]{DeVorePetrovaWojtaszczyk2013}}]
\label[theorem]{thm:weak_greedy_convergence_rates}
Suppose there exists a \emph{solution map} $u:\prmtrSet\to\mathcal H$,
such that the solution set $\mathcal M=u(\prmtrSet)$ is compact and contained
in the closed unit ball of~$\mathcal H$.
Let $(\prmtr_r)_{r\in \N}$ be a sequence of snapshot parameters and 
$(V_r)_{r\in \N}$  a sequence of RB spaces, 
so that
$V_{r+1}
=
V_r+\sspan_{\R}\{u(\prmtr_{r})\}.$
Moreover, assume that the following conditions hold:
\begin{enumerate}
\item \label{item:weak_greedy}
\emph{(Weak greedy selection)}
Suppose that there exists a constant $\vartheta\in(0,1]$,
called the \emph{weakness parameter}, 
such that
\begin{equation*}
\dist(u(\prmtr_r), V_r)
\, \geq \, 
\vartheta
\sup_{\prmtr\in\prmtrSet}
\dist(u(\prmtr), V_r)
\qquad \text{for all }
r \in \N.
\end{equation*}
\item \label{item:rate_decay} 
\emph{(Decaying Kolmogorov-$n$-width)}
There exist constants $\alpha_0, \alpha_1,\alpha_2>0$, such that
\[
d_n(\mathcal M)
\, \leq \,
\alpha_0 \exp \bigl(-\alpha_1 n^{\alpha_2} \bigr)
\qquad
\text{for all }n\in \N,
\]
where $d_n(\mathcal M)$ is the Kolmogorov-$n$-width of $\mathcal{M}$ from \Cref{def:sol_set}.

\end{enumerate}
Then, with $C_1 \coloneqq \sqrt{2 \alpha_0} / \vartheta$ and
$C_2 \coloneqq 2^{-1-2\alpha_2}\alpha_1$, 
it holds that
\begin{equation}
\sup_{\prmtr\in\prmtrSet}
\dist\bigl(u(\prmtr),V_{1+r}\bigr)
\,\leq\,
C_1\exp\bigl(-C_2\,  r^{\alpha_2}\bigr)
\qquad\text{for all }r\in\N.
\label{eq:weak_greedy_exponential_rate}
\end{equation}
In particular, the greedy procedure converges at an exponential rate if $\alpha_2=1$, 
or at a sub-exponential rate if $0<\alpha_2<1$.
\end{theorem}

Our formulation here concerns only (sub-)exponential width decay of the form
$d_n(\mathcal M)\leq \alpha_0\exp(-\alpha_1n^{\alpha_2})$ for $\alpha\in(0,1]$.
We refer to \cite{DeVorePetrovaWojtaszczyk2013} for the proof of this theorem.

\begin{remark}[Fixed initial reduced space]
\label[remark]{rem:fixed_initial_space_weak_greedy}
The standard weak greedy theorem is commonly formulated for a sequence starting from the zero space.
The formulation above also covers an arbitrary fixed initial space~$V_1$.
To see this, 
let $P_1^\perp$ be the orthogonal projector onto $V_1^\perp$ and 
set $\tilde{\mathcal M}\coloneqq P_1^\perp\mathcal M$.
For every space $V_r$ containing~$V_1$, we have
\begin{equation*}
\dist(u(\prmtr),V_r)
=
\dist\bigl(P_1^\perp u(\prmtr),P_1^\perp V_r\bigr)
\qquad \text{and} \qquad
d_n(\tilde{\mathcal M})
\leq
d_n(\mathcal M).
\end{equation*}
Hence, the standard theorem can be applied to the projected greedy sequence in~$V_1^\perp$.
This explains the index shift on the left hand side of  \eqref{eq:weak_greedy_exponential_rate}.
\end{remark}

Regarding snapshot selections, 
the theoretically ideal choice of the next snapshot parameter $\prmtr_r$ 
would maximize
$\prmtr\mapsto\dist\bigl(u(\prmtr),V_r\bigr)$ over~$\prmtrSet$.
This corresponds to $\vartheta = 1$ in \ref{item:weak_greedy}, 
which is called the ``strong greedy'' variant, 
and is usually too expensive to carry out. 
An idealized weak greedy alternative is to maximize an error indicator function
$\Delta^{(r)}:\prmtrSet \to [0,\infty)$ over the full parameter set,
which should be much cheaper to evaluate than
$\dist\bigl(u(\cdot),V_r\bigr)$.
In contrast, the implemented procedure from \Cref{alg:greedy_reduced_basis}
performs this indicator maximization only over the finite set~$\Xi_{\mathrm{train}}$.
The following sufficient condition applies to either choice of the maximization set,
although our formulation only concerns $\prmtrSet$ for brevity.

\begin{lemma}[Error equivalence sufficient for weak greedy selection]
\label{lem:error_equivalence_weak_greedy}
For each $r\in \mathbb{N}$, suppose there exists $\Delta^{(r)}:\prmtrSet \to [0,\infty)$,
such that 
\begin{equation}
\label{eq:indicator_projection_error_equivalence}
c \, \Delta^{(r)}(\prmtr)
\, \leq \,
\dist(u(\prmtr), V_r)^2
\, \leq \,
C \, \Delta^{(r)}(\prmtr)
\qquad \text{for all }
\prmtr \in \prmtrSet,
\end{equation} 
where $c,C>0$ are constants independent of $r$ and of $\prmtr$.
Then, the sequence $(\prmtr_r)_{r \in \mathbb{N}}$ selected via 
$\prmtr_r \, \in \, \operatorname*{arg\,max}_{\prmtr \in \prmtrSet} \Delta^{(r)}(\prmtr)$
satisfies the
condition \Cref{thm:weak_greedy_convergence_rates}~\ref{item:weak_greedy} with $\vartheta = \sqrt{c / C} $. 
\end{lemma}

An error indicator satisfying the left hand side of
\eqref{eq:indicator_projection_error_equivalence} is called \emph{efficient},
whereas one satisfying the right hand side is called \emph{reliable}. 
The proof of \Cref{lem:error_equivalence_weak_greedy}
is straightforward and omitted.

\subsubsection{Compact solution set and valid weak greedy selection rules}

To apply \Cref{thm:weak_greedy_convergence_rates} to our parametric Lindbladian context,
the existence of a solution map with a compact solution set is necessary.
For the remainder of the convergence analysis, we impose the following assumption.

\begin{assumption}[Continuity and uniform well-posedness]
\label[assumption]{ass:continuous_wellposedness}
The following conditions hold:
\begin{enumerate}
\item
The coefficient functions
$\theta_1,\ldots,\theta_\NrHam,\phi_1,\ldots,\phi_\NrDis$
from \Cref{ass:affine_decomp_param_Lindbladian}
are continuous on~$\prmtrSet$. 
\item
$\Lindbladian(\prmtr)$ has a unique steady state for every
$\prmtr\in\prmtrSet$. 
\end{enumerate}
\end{assumption}

For $\prmtr\in\prmtrSet$, define the restricted operator 
\begin{equation}
\Lindbladian_0(\prmtr)
\coloneqq
\Lindbladian(\prmtr)|_{A_0}:A_0\to A_0.
\label{eq:restricted_Lindbladian_A0}
\end{equation}
By \Cref{rem:steady_state_eqn_as_LES} and
\eqref{eq:equivalence_unique_steady_state},
\Cref{ass:continuous_wellposedness} implies 
that $\Lindbladian_0(\prmtr)$ is invertible
and that
\begin{equation}
\SteadyState(\prmtr)
=
I/N
-
\Lindbladian_0(\prmtr)^{-1}\Lindbladian(\prmtr)I / N.
\label{eq:steady_state_solution_map_formula}
\end{equation}

\begin{lemma}[Sufficient condition for a compact solution set]
\label{lem:compact_solution_set}
Under \Cref{ass:continuous_wellposedness},
the solution map $u=\SteadyState:\prmtrSet\to\mathcal H$ is well-defined and continuous,
and the solution set $\mathcal M=\SteadyState(\prmtrSet)$ is compact. 
\end{lemma}

\begin{proof}
By \Cref{ass:affine_decomp_param_Lindbladian}
and \Cref{ass:continuous_wellposedness},
both $\Lindbladian:\prmtrSet\to\mathcal B(\mathcal H)$ and
$\Lindbladian_0:\prmtrSet\to\mathcal B(A_0)$ are continuous.
Since inversion is continuous on the set of invertible operators,
\eqref{eq:steady_state_solution_map_formula} implies continuity of
$\SteadyState:\prmtrSet\to\mathcal H$.
Together with compactness of~$\prmtrSet$,
we deduce compactness of $\mathcal M=\SteadyState(\prmtrSet)$.
\end{proof}

The continuous well-posedness assumption
from \Cref{ass:continuous_wellposedness}
is actually sufficient to show that the two indicators
$\Delta^{(r)}_\mathrm{eig}$ and $\Delta^{(r)}_\mathrm{aff}$
are both efficient and reliable.
To see this, we recall that $\prmtrSet$ is assumed to be compact,
and that $\mathcal{A}(\prmtr) = \Lindbladian(\prmtr)^\dagger \Lindbladian(\prmtr) $
as well as $\gamma(\prmtr)$ and $\beta(\prmtr)$ are defined in \Cref{subsec:state-error}.

\begin{proposition}[Efficient and reliable snapshot selection rules]
\label{prp:valid_weak_greedy}
Suppose \Cref{ass:continuous_wellposedness} holds,
and let $V_r\subseteq\mathcal H$ be a reduced space satisfying $I\in V_r$. 
Then, the constants 
\begin{align*}
\varrho_{\max}
&\coloneqq
\max_{\prmtr\in\prmtrSet}
\max_{\|X\|=1}
\langle X,\mathcal A(\prmtr)X\rangle,
&
\gamma_{\min}
&\coloneqq
\min_{\prmtr\in\prmtrSet}\gamma(\prmtr),
\\
\nu_{\min}
&\coloneqq
\min_{\prmtr\in\prmtrSet}\|\SteadyState(\prmtr)\|^2,
&
\beta_{\min}
&\coloneqq
\min_{\prmtr\in\prmtrSet}\beta(\prmtr).
\end{align*}
 are well-defined and satisfy
$0<\varrho_{\max}<\infty$, $\gamma_{\min}>0$,
$\nu_{\min}>0$, and $\beta_{\min}>0$.
Furthermore, the following properties hold true for any $r\in \mathbb{N}$.
\begin{enumerate}
\item 
\label{item:eig_formulation_error_equivalence} 
For $\Delta^{(r)}_\mathrm{eig}$ from \eqref{eq:eigenvalue_error_indicator},
we have 
\begin{equation*}
\nu_{\min}\varrho_{\max}^{-1}
\Delta^{(r)}_\mathrm{eig}(\prmtr)
\, \leq \, 
\dist(\SteadyState(\prmtr), V_r)^2
\, \leq \, 
\gamma_{\min}^{-1}
\Delta^{(r)}_\mathrm{eig}(\prmtr)
\qquad \text{for all }
\prmtr \in \prmtrSet.
\end{equation*}
\item 
\label{item:aff_formulation_error_equivalence} 
For $\Delta^{(r)}_\mathrm{aff}$ from \eqref{eq:affine_error_indicator},
we have  
\begin{equation*}
\varrho_{\max}^{-1}
\Delta^{(r)}_\mathrm{aff}(\prmtr)
\, \leq \, 
\dist(\SteadyState(\prmtr), V_r)^2
\, \leq \, 
\beta_{\min}^{-2}
\Delta^{(r)}_\mathrm{aff}(\prmtr)
\qquad \text{for all }
\prmtr \in \prmtrSet.
\end{equation*}
\end{enumerate}
\end{proposition}

\begin{proof}
The claims on $\varrho_{\max}$ and $\gamma_{\min}$ follow from the
variational characterization and continuity of the ordered eigenvalues.
Continuity of the steady-state map from \Cref{lem:compact_solution_set} and
compactness of $\prmtrSet$ show that $\nu_{\min}$ is attained; moreover,
$\nu_{\min}\geq1/N$ by \Cref{lem:geometric_properties}.  
Finally,
$\Lindbladian_0(\prmtr)$ is continuously invertible 
under \Cref{ass:continuous_wellposedness} and
$\beta(\prmtr)
=
\|\Lindbladian_0(\prmtr)^{-1}\|_{\mathrm{op}}^{-1}$
with $\|\cdot\|_{\mathrm{op}}$ denoting the operator norm. 
Thus, $\beta$ is continuous and strictly positive, 
and compactness of $\prmtrSet$ yields
$\beta_{\min}>0$.

Let us choose arbitrary $r\in \mathbb{N}$ and $\prmtr \in \prmtrSet$. 
For brevity, we suppress the dependence on $\prmtr$ throughout the remainder of the proof and simply denote
\[
\mathcal L\coloneqq\Lindbladian(\prmtr),
\qquad
\mathcal A\coloneqq\mathcal A(\prmtr),
\qquad
\rho\coloneqq\SteadyState(\prmtr),
\qquad
u\coloneqq {\rho}/{\|\rho\|}.
\]
We also write
$\gamma, \rho_{\mathrm n}^{(r)}$, $\rho_{\mathrm{aff}}^{(r)}$,
$\Delta_{\mathrm{eig}}^{(r)}$ and $\Delta_{\mathrm{aff}}^{(r)}$
for the corresponding quantities at the chosen~$\prmtr$.

Let $P\in\mathcal B(\mathcal H)$ and $P^\perp\in\mathcal B(\mathcal H)$
denote the orthogonal projectors onto $V_r$ and $V_r^\perp$, respectively.
In particular, $\dist(\rho,V_r)=\|P^\perp\rho\|$.
For any $X\in\mathcal H$, let $P_X$ denote the orthogonal projector onto
$\sspan_\mathbb R\{X\}$.
Observe that 
\begin{equation}
\|P_u X\|
\, = \,
\|P_{X} u\|
\, \leq \, 
\|P u\|
\qquad \text{for }
X \in V_r
\text{ with }\|X\|=1,
\label{eq:norm_interchange}
\end{equation}
with equality when $P u \in \sspan\{X\}$, 
i.e., when $Pu\neq 0$ and $X=\pm Pu/\|Pu\|$ 
due to the norm constraint for $X$. 
We also have 
$\ker\mathcal L=\sspan_{\mathbb R}\{ \rho \}$ and  
\begin{equation}
\dist(X,\ker\mathcal L)^2 
=
\|P_u^\perp X\|^2 
=
1-\|P_u X\|^2
\qquad \text{for any }\|X\|=1.
\label{eq:normalized_projection_error_relation}
\end{equation}
Note that $\mathcal{A}P_{ u } = \mathcal{L}^\dagger \mathcal{L}P_{u }= 0$.
Hence, 
we use $X = P_{u } X + P_{u }^\perp X$ and the definition of $\gamma_{\min}$ to deduce 
\begin{equation}
\label{eq:gap_ker_dist_relation}
    \langle X, \mathcal A X\rangle
=
\langle P_{ u }^\perp X, \mathcal A P_{ u }^\perp X\rangle
\geq
\gamma_{\min} \|P_{ u }^\perp X\|^2
=
\gamma_{\min}\,
(1-\|P_u X\|^2)
\qquad \text{for any }\|X\|=1.
\end{equation}

\noindent $\bullet$ 
\emph{We now prove \ref{item:eig_formulation_error_equivalence}}. 
Consider $\rho^{(r)}_{\mathrm{n}}$ from \eqref{eq:norm_normalized_reduced_eigenvector}.
Note that 
$\Delta^{(r)}_\mathrm{eig} = \langle \rho^{(r)}_{\mathrm{n}}, \mathcal{A}\rho^{(r)}_{\mathrm{n}} \rangle$
and 
$\|\rho^{(r)}_{\mathrm{n}}\|=1$. 
For every $X\in V_r$ with $\|X\|=1$, 
\eqref{eq:gap_ker_dist_relation} 
and \eqref{eq:norm_interchange} yield
\[
\begin{aligned}
\langle X,\mathcal A X\rangle
\geq
\gamma_{\min}\big(1-\|P_u X\|^2\big)
\geq
\gamma_{\min}\big(1-\|P u\|^2\big)
=
\gamma_{\min}\dist(u,V_r)^2.
\end{aligned}
\]
Applying this estimate to $\rho^{(r)}_{\mathrm n}$,
and using $\|\rho\|\leq1$ from
\Cref{lem:geometric_properties}, we obtain 
\[
\dist(\rho,V_r)^2
\leq
\dist(u,V_r)^2
\leq
\gamma_{\min}^{-1}
\Delta^{(r)}_{\mathrm{eig}}.
\]

To prove the first inequality,
note that $I\in V_r$ and
$\langle\rho,I\rangle=\tr\rho=1$.
Hence, $P\rho\neq0$,
and 
$ X_\rho \coloneqq P\rho/\|P\rho\| \in V_r$ is admissible in \eqref{eq:abstract_EVP}.
Applying \eqref{eq:norm_interchange} and
\eqref{eq:normalized_projection_error_relation},
we deduce
$$
\|P_u^\perp X_\rho\|^2
= 1-\|P_u X_\rho\|^2 = 1-\|P_{X_\rho }u\|^2
= 1-\|P u\|^2
= \dist(u,V_r)^2
= \|\rho\|^{-2} {\dist(\rho,V_r)^2}.
$$
Using 
the minimizing property of
$\rho_{\mathrm n}^{(r)}$, 
$ \Lindbladian = \Lindbladian P_u^\perp $, 
 the definitions of
$\varrho_{\max}$ and $\nu_{\min}$,
and the previous identity,
we obtain
\[
\Delta_{\mathrm{eig}}^{(r)}
\leq
\|\Lindbladian X_\rho\|^2
= \|\Lindbladian P_u^\perp X_\rho\|^2
\leq
\varrho_{\max}
\|P_u^\perp X_\rho\|^2
=
{\varrho_{\max}}{\|\rho\|^{-2}}\dist(\rho,V_r)^2
\leq
{\varrho_{\max}}\,{\nu_{\min}^{-1}}\,\dist(\rho,V_r)^2.
\]

\smallskip
\noindent $\bullet$ 
\emph{We now prove \ref{item:aff_formulation_error_equivalence}}. 
Consider $\rho^{(r)}_{\mathrm{aff}}$ from \eqref{eq:reduced_affine_approximation}.
Note that 
$\Delta^{(r)}_\mathrm{aff} = \|\mathcal L\rho^{(r)}_{\mathrm{aff}} \|^2$
and 
$\tr \rho^{(r)}_{\mathrm{aff}} =1$. 
As observed above, $P \rho\in V_r\cap A_1$ is admissible in
\eqref{eq:reduced_affine_residual_problem}. 
Hence, the minimizing property of $\varrho^{(r)}_{\mathrm{aff}}$,
$ \Lindbladian\rho=0$, and the definition of $\rho_{\max}$ imply
\[
\begin{aligned}
\Delta^{(r)}_{\mathrm{aff}}
\leq
\|\mathcal L P \rho\|^2
=
\|\mathcal L(P \rho-\rho)\|^2
\leq
\varrho_{\max}\|P^\perp\rho\|^2
=
\varrho_{\max}\dist(\rho,V_r)^2.
\end{aligned}
\]
This proves the first inequality in
\ref{item:aff_formulation_error_equivalence}. 

Finally, set
$E^{(r)}_{\mathrm{aff}}
\coloneqq
\rho^{(r)}_{\mathrm{aff}}-\rho$.
Since both matrices have unit trace, $E^{(r)}_{\mathrm{aff}}\in A_0$.
By the definitions of $\beta(\prmtr)$ and $\beta_{\min}$,
it holds that  $\Delta^{(r)}_{\mathrm{aff}}
=
\|\mathcal L E^{(r)}_{\mathrm{aff}}\|^2
\geq
\beta_{\min}^2\|E^{(r)}_{\mathrm{aff}}\|^2.$
Since $\rho^{(r)}_{\mathrm{aff}}\in V_r$, 
we have 
$\dist(\rho,V_r)^2
\leq \|\rho^{(r)}_{\mathrm{aff}}-\rho\|^2 
= 
\| E^{(r)}_{\mathrm{aff}} \|^2$.
Combining these two estimates 
yields the desired inequality.
\end{proof}

\subsection{Analytic regularity, width decay and convergence of the greedy algorithm}
\label{subsec:width-convergence}

The preceding subsection verified the weak greedy selection condition
from \Cref{thm:weak_greedy_convergence_rates}~\ref{item:weak_greedy} for both reduced formulations.
It remains to establish the decay of the Kolmogorov widths required in
\ref{item:rate_decay}.

To this end, we shall impose additional assumptions on the parameter dependence and on the geometry of the parameter domain.
We first establish analytic regularity of the steady-state map,
then derive sub-exponential width decay by tensorized Chebyshev interpolation,
and finally combine these results with the weak greedy framework.
We conclude by relating the resulting idealized full-domain estimate
to the finite training set used in \Cref{alg:greedy_reduced_basis}.

\begin{assumption}[Real-analytic parameter dependence]
\label[assumption]{ass:real_analytic_parameter_dependence}
There exists an open neighborhood $\Omega\subseteq\R^\pDim$ of $\prmtrSet$,
such that each coefficient function
$\theta_1,\ldots,\theta_\NrHam,\phi_1,\ldots,\phi_\NrDis$
from \Cref{ass:affine_decomp_param_Lindbladian}
admits a real-analytic extension to~$\Omega$.
\end{assumption}

\Cref{ass:real_analytic_parameter_dependence} strengthens the continuity requirement
from \Cref{ass:continuous_wellposedness}.
The extended functions $\phi_\IdxDis$ are not required to remain nonnegative outside~$\prmtrSet$,
as their extensions are only used to establish regularity of the solution map
in a neighborhood of the parameter set.

\begin{proposition}[Analyticity of the steady-state map]
\label[proposition]{prop:analyticity_steady_state_map}
Suppose \Cref{ass:continuous_wellposedness}
and \Cref{ass:real_analytic_parameter_dependence} hold.
Then the steady-state map
$\SteadyState:\prmtrSet\to\mathcal H$ is real-analytic,
in the sense that it admits an $\mathcal H$-valued real-analytic extension
to an open neighborhood $\Omega_0\subseteq\Omega$ of~$\prmtrSet$.
\end{proposition}

\begin{proof}
By \Cref{ass:real_analytic_parameter_dependence} and the affine decomposition
\eqref{eq:Lindbladian_affine_decomposition},
$\Lindbladian$ admits a real-analytic extension to~$\Omega$,
which we denote again by~$\Lindbladian$.
The definition of $\Lindbladian_0$ from \eqref{eq:restricted_Lindbladian_A0}
extends to~$\Omega$ by the same restriction formula,
and the resulting operator-valued map is real-analytic.
By \Cref{ass:continuous_wellposedness} and \eqref{eq:equivalence_unique_steady_state},
$\Lindbladian_0(\prmtr)$ is invertible for every $\prmtr\in\prmtrSet$.
Since the set of invertible operators on~$A_0$ is open,
\(
\Omega_0
\coloneqq
\{
\prmtr\in\Omega:
\Lindbladian_0(\prmtr)\text{ is invertible}
\}
\)
is an open neighborhood of~$\prmtrSet$.

The inversion map is real-analytic on the set of invertible operators.
Hence, the right hand side of \eqref{eq:steady_state_solution_map_formula}
defines an $\mathcal H$-valued real-analytic map on~$\Omega_0$.
On~$\prmtrSet$, it coincides with the steady-state map,
and therefore, provides the asserted extension.
\end{proof}

We next impose an additional geometric structure on the parameter domain.
This allows us to turn the analytic regularity from
\Cref{prop:analyticity_steady_state_map} into a quantitative approximation result.

\begin{proposition}[Exponential or sub-exponential width decay]
\label[proposition]{prop:analytic_width_decay}
Suppose the assumptions of \Cref{prop:analyticity_steady_state_map} hold,
and suppose additionally that
\begin{equation*}
\prmtrSet
=
[a_1,b_1]\times \cdots \times[a_\pDim,b_\pDim]
\qquad\text{with}\qquad
a_j<b_j
\quad\text{for }j=1,\ldots,\pDim.
\end{equation*}
Then there exist constants $\alpha_0,\alpha_1>0$ such that
\begin{equation}
d_n(\mathcal M)
\,\leq\,
\alpha_0\exp\Bigl(-\alpha_1 n^{1/\pDim}\Bigr)
\qquad\text{for all }n\in\N.
\label{eq:generalized_exponential_width_decay}
\end{equation}
\end{proposition}
In case $\pDim=1$, this claim can be found in \cite[Sec.~3.4.1]{SirK16}.

\begin{proof}
Let $T:\mathbb{R}^{\pDim}\to\mathbb{R}^{\pDim}$ be an affine bijection
satisfying $T([-1,1]^{\pDim})=\prmtrSet$,
and set $F\coloneqq\SteadyState\circ T$.
By \Cref{prop:analyticity_steady_state_map},
$F$ is real-analytic on an open neighborhood of the compact cube
$[-1,1]^{\pDim}$.
By the standard complexification result 
for finite-dimensional real-analytic maps, 
cf.~\cite[Sec.~2.2]{KrantzParks2002}, 
$F$ admits a holomorphic extension 
to a complex neighborhood of the cube, 
with values in 
$\mathcal H \otimes \mathbb{C} = \mathbb{C}^{\fdim \times \fdim}$.
Hence, after possibly reducing this neighborhood,
there exists $R>1$, 
such that this extension is bounded 
and holomorphic on the product Bernstein polyellipse $\mathcal E_R^{\pDim} \subseteq \mathbb{C}^{\pDim}$,
where $\mathcal E_R\subseteq\C$ denotes the Bernstein ellipse with parameter~$R$,
cf.~\cite[Sec.~2.2]{KrantzParks2002}.

For every integer $m\geq0$, 
let $\mathcal I_mF$ be the tensorized Chebyshev-Lobatto interpolant
of degree~$m$ in each parameter direction.
The standard interpolation estimate for functions that are holomorphic on a
Bernstein polyellipse gives
\begin{equation}
\sup_{y\in[-1,1]^{\pDim}}
\|
F(y)-\mathcal I_mF(y)
\|
\,\leq\,
C_0R^{-m}
\label{eq:vector_chebyshev_interpolation_error}
\end{equation}
for some constant $C_0>0$ independent of~$m$;
see, e.g., \cite{NobileTamelliniTempone2016,ChkifaCohenSchwab2014,AdcockBrugiapagliaDexterMoraga2024}
for polynomial approximation and interpolation of Hilbert- or Banach-space-valued maps.

Let $(y_\nu)_{\nu\in\{0,\ldots,m\}^{\pDim}}$ be the tensorized Chebyshev nodes.
The Lagrange representation of the interpolant
with respect to a Lagrangian basis 
$\{\ell_\nu\}_{\nu\in\{0,\ldots,m\}^\pDim}$ 
reads
\begin{equation*}
\mathcal I_mF(y)
\, = \!
\sum_{\nu\in\{0,\ldots,m\}^{\pDim}}
\ell_\nu(y)F(y_\nu).
\end{equation*}
Consequently, $\mathcal I_mF(y)$ belongs to the snapshot space
$\sspan_\R
\bigl\{
F(y_\nu):
\nu\in\{0,\ldots,m\}^{\pDim}
\bigr\},$
which has dimension at most $(m+1)^{\pDim}$.
Together with \eqref{eq:vector_chebyshev_interpolation_error}, this yields
\begin{equation}
d_{(m+1)^{\pDim}}(\mathcal M)
\,\leq\,
C_0R^{-m}.
\label{eq:width_at_tensor_degrees}
\end{equation}

For arbitrary $n\in\N$, choose
$m_n\coloneqq\lfloor n^{1/\pDim}\rfloor-1$.
Then $m_n\geq0$ and $(m_n+1)^{\pDim}\leq n$.
By monotonicity of the Kolmogorov widths,
\eqref{eq:width_at_tensor_degrees}
and using
$\lfloor n^{1/\pDim}\rfloor\geq n^{1/\pDim}-1$,
we obtain
\[
d_n(\mathcal M)
\,\leq\,
C_0R^{-m_n}
=
C_0R^{1-\lfloor n^{1/\pDim}\rfloor}
\,\leq\,
C_0R^2\exp\Bigl(-(\log R)n^{1/\pDim}\Bigr).
\]
This proves \eqref{eq:generalized_exponential_width_decay}
with $\alpha_0\coloneqq C_0R^2$ and $\alpha_1\coloneqq\log R$.
\end{proof}

\begin{remark}[Analytically parametrized parameter domains]
\label[remark]{rem:analytic_parametrization_width_decay}
The hyperrectangle assumption can be replaced by an analytic parametrization.
More precisely, suppose that $Q\subseteq\R^{\ell}$ is a compact hyperrectangle
and that $\Phi:Q\to\prmtrSet$ is surjective.
If $\SteadyState\circ\Phi$ admits a real-analytic extension to a neighborhood of~$Q$,
then the proof of \Cref{prop:analytic_width_decay} applies to this composition
and gives constants $\alpha_0,\alpha_1>0$, 
such that
$d_n(\mathcal M)
\,\leq\,
\alpha_0\exp\bigl(-\alpha_1 n^{1/\ell}\bigr).$
\end{remark}

With \Cref{prop:analytic_width_decay}, we are in a position to conclude the convergence analysis 
for the idealized greedy procedures, 
in which the error indicator is maximized over the full parameter set.

\begin{corollary}[Convergence of the idealized greedy procedures]
\label[corollary]{cor:idealized_greedy_convergence}
Suppose the assumptions of \Cref{prop:analytic_width_decay} hold.
For $\star\in\{\mathrm{eig},\mathrm{aff}\}$,
let $V_{1,\star}\coloneqq\sspan_\R\{I\}$ and suppose the sequence of reduced spaces is generated by
\begin{equation*}
\prmtr_{r,\star}
\,\in\,
\operatorname*{arg\,max}_{\prmtr\in\prmtrSet}
\Delta_\star^{(r)}(\prmtr),
\qquad
V_{r+1,\star}
=
V_{r,\star}
+
\sspan_\R\{\SteadyState(\prmtr_{r,\star})\}
.
\end{equation*}
Then there exist $C_1,C_2>0$, 
such that
\begin{equation*}
\sup_{\prmtr\in\prmtrSet}
\dist\bigl(\SteadyState(\prmtr),V_{1+r,\star}\bigr)
\,\leq\,
C_1\exp\Bigl(-C_2 \, r^{1/\pDim}\Bigr)
\qquad\text{for all }r\in\N.
\end{equation*}
\end{corollary}

\begin{proof}
By \Cref{prp:valid_weak_greedy} and \Cref{lem:error_equivalence_weak_greedy},
the two indicators generate weak greedy sequences with respective weakness
parameters
$\vartheta_{\mathrm{eig}}
 \coloneqq 
\sqrt{{\nu_{\min}\gamma_{\min}}/{\varrho_{\max}}}$
and 
$\vartheta_{\mathrm{aff}}
\coloneqq 
{\beta_{\min}}/{\sqrt{\varrho_{\max}}}$,
respectively.
By \Cref{prop:analytic_width_decay},
$d_n(\mathcal M)$ satisfies
\Cref{thm:weak_greedy_convergence_rates}~\ref{item:rate_decay}
with $\alpha_2=1/\pDim$,
and the assertion follows from \Cref{thm:weak_greedy_convergence_rates}.
\end{proof}

Finally, let us discuss the finite training set used in the actual implementation.
The influence of finite and randomly sampled training sets on reduced basis greedy procedures
is considered in  \cite{HesthavenStammZhang2014}.

To quantify the discretization error introduced by a nonempty finite training set $\Xi_{\mathrm{train}}\subseteq\prmtrSet$, 
we define 
\begin{align}
\mathcal M_\Xi
&\coloneqq
\SteadyState(\Xi_{\mathrm{train}}),
&
\delta_\Xi
&\coloneqq
\sup_{\prmtr\in\prmtrSet}
\min_{\xi\in\Xi_{\mathrm{train}}}
\|
\SteadyState(\prmtr)-\SteadyState(\xi)
\|,
\label{eq:intrinsic_training_set_distance}\\
h_\Xi
&\coloneqq
\sup_{\prmtr\in\prmtrSet}
\min_{\xi\in\Xi_{\mathrm{train}}}
\|\prmtr-\xi\|_2,
&
L_u
&\coloneqq
\sup_{\prmtr\in\prmtrSet}
\|
D\SteadyState(\prmtr)
\|_{\mathrm{op}},
\label{eq:training_set_fill_distance_lipschitz_constant}
\end{align}
where $D$ denotes the Fréchet derivative, 
$\|\cdot\|_2$ the Euclidean norm 
 and $\|\cdot\|_{\mathrm{op}}$ the operator norm.

\begin{proposition}[Convergence for a finite training set]
\label[proposition]{prop:finite_training_set_convergence}
Suppose the assumptions of \Cref{prop:analytic_width_decay} hold,
and let $\Xi_{\mathrm{train}}\subseteq\prmtrSet$ be nonempty and finite.
For either formulation $\star\in\{\mathrm{eig},\mathrm{aff}\}$, let
$V_{1+k,\star}^{\Xi}$ denote the space produced by
\Cref{alg:greedy_reduced_basis} after $k$ snapshot enrichments.  
Then, there exist constants $C_1,C_2>0$, 
independent of~$r$ and of the training set, 
such that
\begin{equation}
\sup_{\xi\in\Xi_{\mathrm{train}}}
\dist\bigl(\SteadyState(\xi),V_{1+r,\star}^{\Xi}\bigr)
\,\leq\,
C_1\exp\Bigl(-C_2 r^{1/\pDim}\Bigr)
\label{eq:discrete_training_manifold_convergence}
\end{equation}
for every snapshot enrichment performed by the algorithm.
If the algorithm terminates because $\Xi_{\mathrm{act}}=\emptyset$,
then its final space~$V_{r,\star}^{\Xi}$ additionally satisfies
\begin{equation}
\sup_{\prmtr\in\prmtrSet}
\dist\bigl(\SteadyState(\prmtr),V_{r,\star}^{\Xi}\bigr)
\,\leq\,
\sqrt{C_\star\varepsilon_{\mathrm{tol}}}
+
\delta_\Xi
\,\leq\,
\sqrt{C_\star\varepsilon_{\mathrm{tol}}}
+
L_u h_\Xi,
\label{eq:finite_training_set_stopping_error}
\end{equation}
where
$C_\mathrm{eig}\coloneqq\gamma_{\min}^{-1}$ and
$C_\mathrm{aff}\coloneqq\beta_{\min}^{-2}$.
\end{proposition}

\begin{proof}
First, observe that the implementation of the active training set $\Xi_{\mathrm{act}}$ 
in \Cref{alg:greedy_reduced_basis} 
selects a maximizer over the full training set $\Xi_{\mathrm{train}}$ at every enrichment: 
once a parameter is removed from~$\Xi_{\mathrm{act}}$,
its indicator is bounded by~$\varepsilon_{\mathrm{tol}}$ and remains so,
because the indicators are monotonically nonincreasing with~$r$.
As long as $\Xi_{\mathrm{act}}$ is nonempty, 
every active parameter has indicator value
strictly larger than~$\varepsilon_{\mathrm{tol}}$.
Hence, the maximizer over the active set is also a maximizer over~$\Xi_{\mathrm{train}}$.

The error equivalences from \Cref{prp:valid_weak_greedy} remain valid after restricting
the parameter set to~$\Xi_{\mathrm{train}}$.
Hence, for each formulation the algorithm is a weak greedy procedure
for the finite set~$\mathcal M_\Xi$ with the corresponding weakness parameter
as given in the proof of \Cref{cor:idealized_greedy_convergence}.
Since $\mathcal M_\Xi\subseteq\mathcal M$, 
we have
$d_n(\mathcal M_\Xi)
\,\leq\,
d_n(\mathcal M)
\,\leq\,
\alpha_0\exp\bigl(-\alpha_1 n^{1/\pDim}\bigr)$
for all $n\in\N$.
Hence, the weak greedy theorem, 
together with the fixed-initial-space argument from
\Cref{rem:fixed_initial_space_weak_greedy},
 proves \eqref{eq:discrete_training_manifold_convergence}.

For any subspace $V\subseteq\mathcal H$ and $\prmtr\in\prmtrSet$, choose
$\xi\in\Xi_{\mathrm{train}}$ minimizing the state-space distance in
\eqref{eq:intrinsic_training_set_distance}.  Since the distance to a fixed
subspace is $1$-Lipschitz, we have 
\[
\dist\bigl(\SteadyState(\prmtr),V\bigr)
\leq
\dist\bigl(\SteadyState(\xi),V\bigr)
+
\|\SteadyState(\prmtr)-\SteadyState(\xi)\|.
\]
If $\Xi_{\mathrm{act}}$ is empty, monotonicity of $\Delta_\star^{(r)}$ implies
\(
\max_{\xi\in\Xi_{\mathrm{train}}}
\Delta_\star^{(r)}(\xi)
\,\leq\,
\varepsilon_{\mathrm{tol}}.
\)
The reliability inequalities from \Cref{prp:valid_weak_greedy} then show that
$\sup_{\xi\in\Xi_{\mathrm{train}}}
\dist\bigl(\SteadyState(\xi),V_{r,\star}^{\Xi}\bigr)
\,\leq\,
\sqrt{C_\star\varepsilon_{\mathrm{tol}}}.$
This estimate and the preceding discretization argument yield 
the first inequality in \eqref{eq:finite_training_set_stopping_error}.

It remains to prove that $\delta_\Xi\leq L_uh_\Xi$.
By \Cref{prop:analyticity_steady_state_map},
$D\SteadyState$ is continuous on the compact set~$\prmtrSet$,
so $L_u<\infty$.
Moreover, the hyperrectangle~$\prmtrSet$ is convex.
Hence, the line segment between $\prmtr$ and~$\xi$ lies in~$\prmtrSet$, 
and the mean-value estimate gives
\(
\|
\SteadyState(\prmtr)-\SteadyState(\xi)
\|
\,\leq\,
L_u\|\prmtr-\xi\|_2.
\)
Taking the supremum proves $\delta_\Xi\leq L_u h_\Xi$,
 and hence, 
 the remaining
inequality in \eqref{eq:finite_training_set_stopping_error}.
\end{proof}

\section{Numerical experiments}\label{sec:num}

This section investigates the approximation properties and 
computational performance of the RB methods 
from \Cref{sec:RBM}.
After introducing the test problems and numerical protocol 
in \Cref{subsec:num_setup}, 
we study empirical solution-set compressibility and
its system-size dependence 
in \Cref{subsec:num_compressibility}.
We examine greedy convergence, RB approximation accuracy,
and offline construction costs
in \Cref{subsec:num_greedy_convergence},  
assess the residual indicators and state-error estimates 
in \Cref{subsec:num_error_estimates},
and 
quantify the online computational costs 
and demonstrate the reduced evaluation of physical observables in \Cref{subsec:num_performance}.

Our implementation is conducted in Python, 
and, together with the data and scripts to reproduce the numerical results,
are available at \url{https://doi.org/10.18419/DARUS-6519}.

\subsection{Test problems and computational protocol}
\label{subsec:num_setup}

Following \cite{ChristiansenBaranPaaske2025}, 
we consider two parametric Lindbladians on open chains of length $L$,
both with Hilbert space dimension $\fdim=2^L$.
For $C\in\C^{\fdim\times\fdim}$, we abbreviate
\[
\mathcal D_C(X)
\, \coloneqq \, 
CXC^\dagger-\frac12\{C^\dagger C,X\}.
\]

\subsubsection{Boundary-driven Fermi--Hubbard chain}

Our first test problem is a boundary-driven spinless Fermi--Hubbard (FH)
chain, with one fermionic mode per site.
We consider 
$\prmtr=({\FHinteraction},\delta)
\in
\prmtrSet_{\mathrm{FH}}
\coloneqq
[0,2]\times[-0.25,0.25].$
For $j=1,\ldots,L$, let $c_j$ and $c_j^\dagger$ denote the standard
fermionic annihilation and creation operators at site~$j$, 
satisfying $\{c_j,c_k\}=0$ and $\{c_j,c_k^\dagger\}=\delta_{jk}I$;
see, e.g., \cite[Ch.~2]{AltlandSimons2023}. 
The corresponding occupation operators are given by  
$n_j\coloneqq c_j^\dagger c_j$ and $\widetilde n_j\coloneqq n_j - I/2$
for $j=1,\ldots,L$. 
For a fixed hopping amplitude $t=1$, 
the Hamiltonian reads as 
\begin{equation}
H_{\mathrm{FH}}({\FHinteraction})
=
-t\sum_{j=1}^{L-1}
\left(
c_{j+1}^\dagger c_j+c_j^\dagger c_{j+1}
\right)
+
{\FHinteraction}\sum_{j=1}^{L-1}
\widetilde n_j\widetilde n_{j+1}.
\label{eq:num_fh_hamiltonian}
\end{equation}
With the reference boundary coupling $\Gamma=0.3$, 
the corresponding
Lindbladian is
\begin{equation*}
\begin{aligned}
\mathcal L_{\mathrm{FH}}({\FHinteraction},\delta)X
\,={}&
-\ImgUnit[H_{\mathrm{FH}}( {\FHinteraction}),X]
+
(\Gamma+\delta)\mathcal D_{c_1^\dagger}(X)
+
(\Gamma-\delta)\mathcal D_{c_L}(X).
\end{aligned}
\end{equation*}
Hence, $\mathcal L_{\mathrm{FH}}$ satisfies the affine decomposition
\eqref{eq:Lindbladian_affine_decomposition} with
$\NrHam=\NrDis=2$ and 
 the coefficient functions
$\theta_1=1$, $\theta_2 = {\FHinteraction}$,
$\phi_1 =\Gamma+\delta$ and $\phi_2 =\Gamma-\delta$.

\subsubsection{XYZ chain}

Our second test problem is a dissipative spin-$1/2$ XYZ chain;
see, e.g., \cite{LeeGopalakrishnanLukin2013} and
\cite{PopkovProsenZadnik2020}.
We consider
$\prmtr=(J_x,J_y)
\in
\prmtrSet_{\mathrm{XYZ}}
\coloneqq
[-1,1]^2.$
Let $\sigma^x, \sigma^y, \sigma^z$ denote the Pauli matrices; 
see, e.g., \cite[Ch.~2]{AltlandSimons2023}. 
For $\alpha\in\{x,y,z\}$ and $j=1,\ldots,L$, 
we write
$\sigma_j^\alpha
\coloneqq
I_2^{\otimes(j-1)}
\otimes\sigma^\alpha
\otimes I_2^{\otimes(L-j)}$
and 
$\sigma_j^-
\coloneqq
\left(
\sigma_j^x-\ImgUnit\sigma_j^y
\right)/2.$
With $J_z=1$ fixed, 
the Hamiltonian is given by 
\begin{equation*}
H_{\mathrm{XYZ}}(J_x,J_y)
\, =
\sum_{j=1}^{L-1}
\left(
J_x \sigma_j^x\sigma_{j+1}^x
+
J_y \sigma_j^y\sigma_{j+1}^y
+
J_z \sigma_j^z\sigma_{j+1}^z
\right),
\end{equation*}
and the corresponding Lindbladian reads as
\begin{equation*}
\mathcal L_{\mathrm{XYZ}}(J_x,J_y)X
\, =
-\ImgUnit[H_{\mathrm{XYZ}}(J_x,J_y),X]
+
\sum_{j=1}^L
\mathcal D_{\sigma_j^-}(X).
\end{equation*}
This has the affine form
\eqref{eq:Lindbladian_affine_decomposition} with
$\NrHam=3$, $\NrDis=1$ and $\theta_1=J_x$, $\theta_2 =J_y$,
$\theta_3 = 1$, $\phi_1=1$.

\subsubsection{High-fidelity solver, parameter grids, and error measures}
\label{subsubsec:exp_general_config}

Let $ \mathcal{Q} X\coloneqq\tr(X)I $.  
We assemble the sparse matrix representation of a Lindbladian via vectorization and the Kronecker product,
and solve the trace-lifted normal equation
\begin{equation}
\bigl(
\Lindbladian(\prmtr)^\dagger\Lindbladian(\prmtr)
+
\frac{\alpha}{\fdim} \mathcal{Q}
\bigr)
\SteadyState(\prmtr)
\, = \, 
\frac{\alpha}{\fdim}I 
\label{eq:num_hf_solver}
\end{equation}
by the conjugate-gradient method with relative tolerance $10^{-14}$,
using $\alpha \coloneqq 10$.
Since $\ker \mathcal{Q}=A_0$, 
\eqref{eq:num_hf_solver} having a unique solution
is equivalent to  $\Lindbladian(\prmtr)$ having a unique steady state,
and also to $\Lindbladian(\prmtr)^\dagger\Lindbladian(\prmtr)+\mathcal{Q}$ being positive definite.
To ensure numerical stability, 
the computed solution is symmetrized and normalized to unit trace.

The spectral gap $\gamma(\prmtr)$ and the trace-zero stability constant
$\beta(\prmtr)$ are evaluated by Lanczos methods at the sampled parameters.

POD and both greedy procedures use tensor-product Chebyshev--Lobatto
training grids of size $49^2$ for FH and $61^2$ for XYZ;
errors are evaluated on $13^2$ uniform validation grids.
{These training-grid sizes were selected by refinement:
for the same $L$, coarser grids yield visibly faster singular-value decay,
indicating that the snapshot ensemble is not yet sufficiently resolved.}

At every training parameter, the eigensolver detected a single zero mode
separated from the remaining spectrum.  This confirms uniqueness of the
steady state at the sampled parameters and indicates uniqueness throughout
$\prmtrSet$.

The greedy spaces are enriched until the error indicator 
is below the prescribed tolerance $10^{-13}$, 
or until the prescribed maximal RB dimension is reached.
For the POD spaces used in the method comparison, singular values
below $\varepsilon \coloneqq 3.2\times10^{-7} \approx {10^{-13/2}}$ are discarded.

\subsection{Empirical compressibility and system-size dependence}
\label{subsec:num_compressibility}

To investigate the compressibility of the steady-state solution sets,
let
$\Xi_{\mathrm{train}}=\{\prmtr_1,\ldots,\prmtr_M\}$
be the corresponding training grid of each model.
For fixed $L$, let $\sigma_{1}(L)\geq\sigma_{2}(L)\geq\cdots$ denote the
singular values of the matrix with 
the vectorized steady states
$\operatorname{vec}\SteadyState(\prmtr_j)$ 
for $j=1,\ldots,M$ as columns.
Their decay measures how well the snapshot ensemble can be approximated,
in the mean-square Frobenius norm, by a low-dimensional linear space; see
\Cref{subsubsec:POD_vs_greedy}.

Since both parameter domains are two-dimensional rectangles,
\Cref{prop:analytic_width_decay} gives, for each fixed~$L$, 
a sub-exponential width bound of the form
$\exp(-c\sqrt{n})$.
Accordingly, \Cref{fig:num_snapshot_spectra} plots the normalized
singular values for $L=5,\ldots,8$ against~$\sqrt{n}$.
For both models, the singular values decay over several orders of magnitude,
demonstrating the empirical compressibility of the sampled solution sets.
Their approximately \emph{linear behavior on the logarithmic scale} is
qualitatively consistent with the width bound.
The dashed lines are least-squares fits over
$\sigma_r(L)/\sigma_1(L)\in[10^{-10},10^{-1}]$.
As $L$ increases from $5$ to~$8$, the fitted slopes become less negative,
changing from $-3.39$ to $-1.97$ for the FH chain and from $-1.22$ to
$-0.84$ for the XYZ chain.
Thus, the decay becomes slower with increasing system size and remains
faster for the FH chain than for the XYZ chain.

\begin{figure}[t]
  \centering
  \begin{subfigure}{0.42\textwidth}
    \centering
    \includegraphics[width=\textwidth, trim=0 8mm 0 5mm]{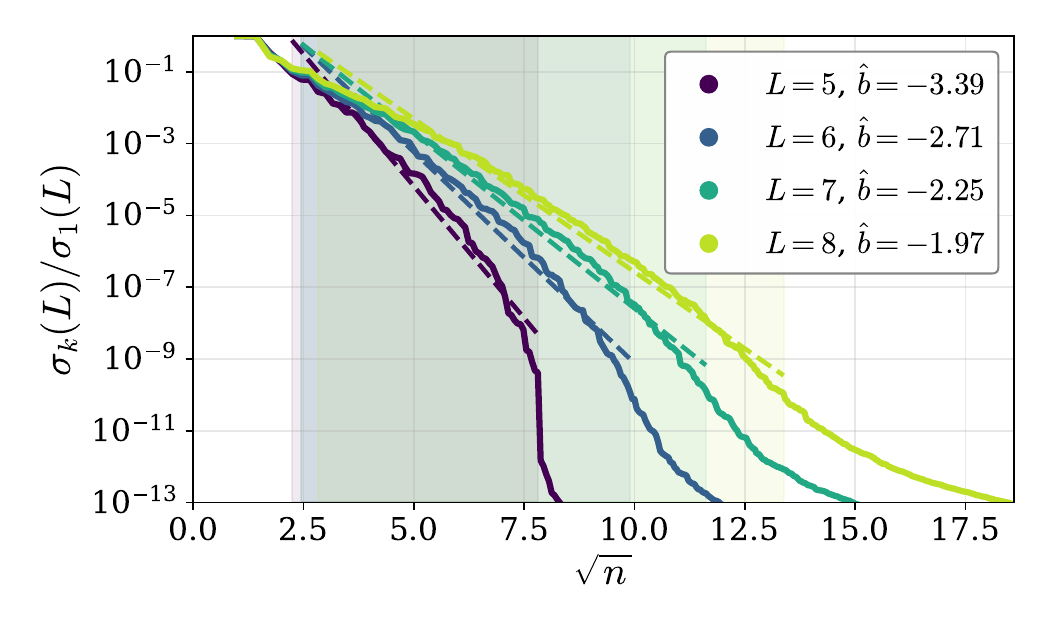}
    \caption{FH chain.}
    \label{fig:num_snapshot_spectrum_sqrt_fh}
  \end{subfigure}
  \qquad
  \begin{subfigure}{0.42\textwidth}
    \centering
    \includegraphics[width=\textwidth, trim=0 8mm 0 5mm]{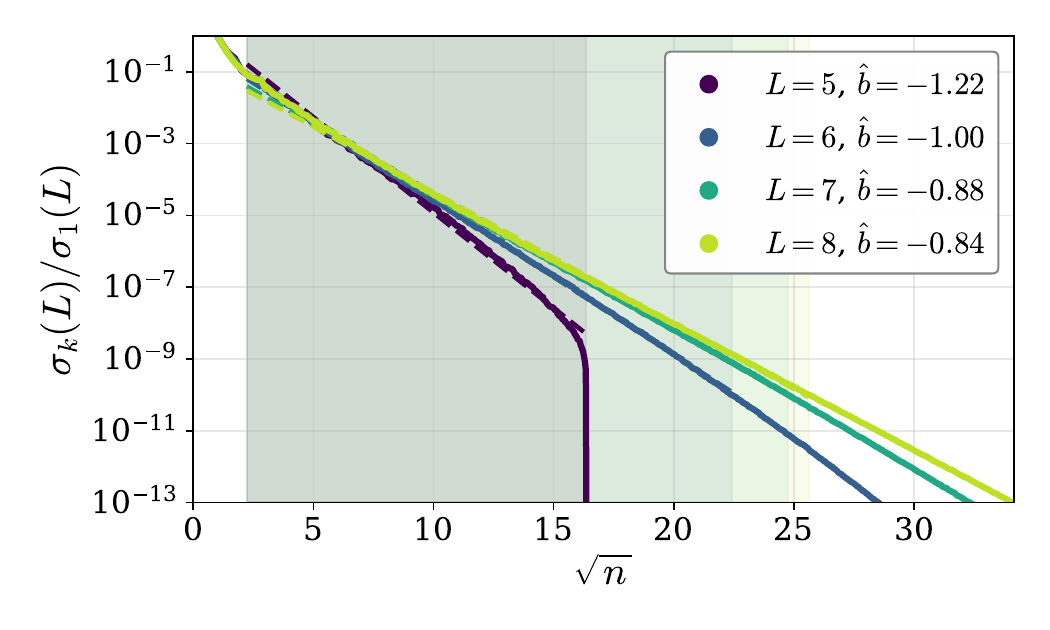}
    \caption{XYZ chain.}
    \label{fig:num_snapshot_spectrum_sqrt_xyz}
  \end{subfigure}
  \caption{Normalized singular values  of the snapshot matrices in $\log$-scale plotted against $\sqrt{n}$.  
  The shaded regions 
  indicate the fitting intervals and the dashed lines 
  the corresponding least-squares fits; 
  $\hat b$ is the fitted slope of
  $\log(\sigma_n(L)/\sigma_1(L))$ against $\sqrt{n}$ over that interval.}
  \label{fig:num_snapshot_spectra}
\end{figure}

For the singular-value cutoff
$\varepsilon=3.2\times10^{-7}$ used in the subsequent computations, 
\Cref{tab:num_truncated_snapshot_rank} reports the retained dimension 
$r_\varepsilon(L) \coloneqq \#\{n:\sigma_{n}(L)\geq\varepsilon\}$ and,
in parentheses, its fraction of the ambient dimension $\fdim^2=4^L$.
Although $r_\varepsilon(L)$ increases with $L$, 
it grows much more slowly
than $\fdim^2$.
Moreover, the XYZ chain consistently requires more modes than the
FH chain, in agreement with the slower singular-value decay
in \Cref{fig:num_snapshot_spectra}.

Overall, both solution sets remain empirically compressible over the considered range of \(L\), 
with the FH chain requiring substantially fewer modes than the XYZ chain.

\begin{table}[t]
  \centering
  \begin{tabular}{c c c c}
    \toprule
    $L$ & $\fdim^2 = 4^L$
    & FH & XYZ \\
    \midrule
    $5$ & $1\,024$  & $50\;(4.88\%)$   & $242\;(23.6\%)$ \\
    $6$ & $4\,096$  & $75\;(1.83\%)$   & $326\;(7.96\%)$ \\
    $7$ & $16\,384$ & $100\;(0.610\%)$ & $381\;(2.33\%)$ \\
    $8$ & $65\,536$ & $127\;(0.194\%)$ & $396\;(0.604\%)$ \\
    \bottomrule
  \end{tabular}
  \caption{ Retained POD dimensions $r_\varepsilon(L)$ for
  $\varepsilon=3.2\times10^{-7}$.  
  The values in parentheses are
  $r_\varepsilon(L)/\fdim^2$.
  }
  \label{tab:num_truncated_snapshot_rank}
\end{table}

\subsection{Offline cost, greedy convergence, and accuracy at fixed RB dimension }
\label{subsec:num_greedy_convergence}

We compare the RB subspaces constructed by EigGreedy and AffGreedy,
with POD as a benchmark.

\subsubsection{Offline cost}
\Cref{tab:num_offline_cost} reports the total computation time for constructing the
final reduced models, including the high-fidelity solves and the assembly
of the reduced model.
For both examples, the greedy procedures substantially reduce the offline
time compared with POD.
This is not surprising, since POD requires the full snapshot ensemble, 
while the greedy procedures only require a small fraction of the high-fidelity solves.

AffGreedy is also faster than EigGreedy.
Its final RB dimension is smaller, and so it requires fewer
high-fidelity solves.
Moreover, the total time spent evaluating the error indicators over the
training grid is $15.5\,\mathrm{s}$ for AffGreedy and
$27.2\,\mathrm{s}$ for EigGreedy in the FH example, and
$728\,\mathrm{s}$ and $1736\,\mathrm{s}$, respectively, in the XYZ example.

\begin{table}[htbp]
  \centering
  \begin{tabular}{c c c c c}
    \toprule
    & $\bar t_{\mathrm{h.f.}}$ & POD & EigGreedy & AffGreedy \\
    \midrule
    FH
    & $12.4\,\mathrm{s}$
    & $9.20\,\mathrm{h}\;(r=127)$
    & $0.476\,\mathrm{h}\;(r=133)$
    & $0.423\,\mathrm{h}\;(r=115)$ \\
    XYZ
    & $7.11\,\mathrm{s}$
    & $8.17\,\mathrm{h}\;(r=396)$
    & $1.56\,\mathrm{h}\;(r=483)$
    & $1.19\,\mathrm{h}\;(r=435)$ \\
    \bottomrule
  \end{tabular}
  \caption{Mean time $\bar t_{\mathrm{h.f.}}$ per high-fidelity solve on the
  validation grid and total offline computation times for the final models at $L=8$, 
  measured serially on the same computer.  
  The final RB dimensions are given
  in parentheses.}
  \label{tab:num_offline_cost}
\end{table}

\subsubsection{Greedy convergence.}
By \Cref{cor:idealized_greedy_convergence} 
and \Cref{prop:finite_training_set_convergence},
both greedy procedures 
possess sub-exponential convergence in the RB dimension. 
To examine this behavior numerically,
\Cref{fig:num_greedy_decay} compares the squared POD singular values
$\sigma_k^2$ with 
$\max_{\prmtr \in \Xi_{\mathrm{train}}}\Delta_{\mathrm{eig}}^{(r)}$ and $\max_{\prmtr  \in \Xi_{\mathrm{train}}}\Delta_{\mathrm{aff}}^{(r)}$
for $L=8$.
The three curves have a similar form and decay over several orders of
magnitude at comparable rates.
This is consistent with the sub-exponential decay predicted for the
two greedy constructions.
As in \Cref{subsec:num_compressibility}, 
we also observe that the decay is considerably faster
for the FH chain than for the XYZ chain.

\begin{figure}[htbp]
  \centering
  \begin{subfigure}{0.32\textwidth}
    \centering
    \includegraphics[width=\textwidth, trim = 0 8mm 0 7mm]{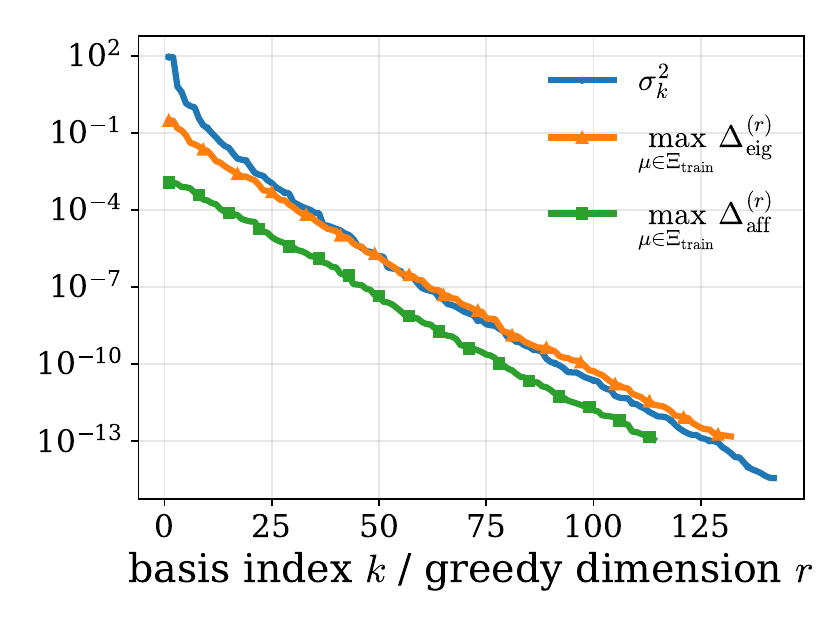}
    \caption{FH chain.}
    \label{fig:num_greedy_decay_fh}
  \end{subfigure}\qquad
  \begin{subfigure}{0.32\textwidth}
    \centering
    \includegraphics[width=\textwidth, trim = 0 8mm 0 7mm]{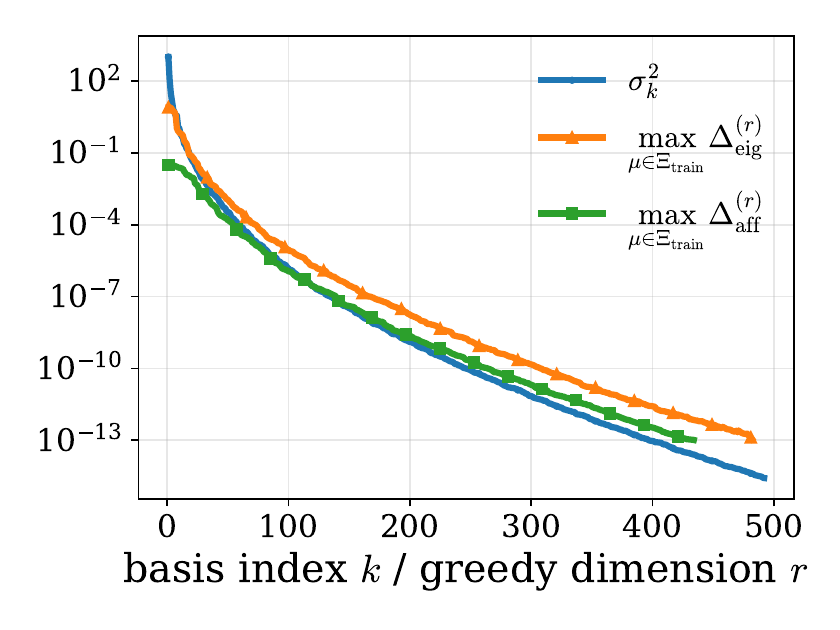}
    \caption{XYZ chain.}
    \label{fig:num_greedy_decay_xyz}
  \end{subfigure}
  \caption{Squared POD singular values and maximal value of the greedy indicators on the training grid for 
  $L=8$.  
  }
  \label{fig:num_greedy_decay}
\end{figure}

The error equivalences in \Cref{prp:valid_weak_greedy} relate both greedy
indicators to the squared projection error.
Therefore, we compare the maximal projection errors on the training grids at
common RB dimensions in \Cref{fig:num_fixed_rank_projection}.
The {RB dimension} is increased in steps of $10$ for the FH chain and $30$
for the XYZ chain, up to the smallest final dimension of the three
constructions, namely $115$ and $396$, respectively.
POD gives the smallest projection errors in all cases.
Between the greedy spaces, AffGreedy generally gives the smaller absolute
projection error, while EigGreedy gives the smaller relative projection error
at the larger {RB dimensions}.
The validation-grid curves have comparable magnitudes and show the same
behavior.
This reversal between EigGreedy and AffGreedy is consistent with the different normalizations in the selection rules:
EigGreedy minimizes the residual of a unit-Frobenius-norm state and is thus
more closely aligned with relative accuracy, whereas AffGreedy imposes unit
trace and measures the residual on the scale of the physical state.

\begin{figure}[htbp]
  \centering
  \begin{subfigure}{0.24\textwidth}
    \centering
    \includegraphics[width=\textwidth, trim = 7mm 12mm 7mm 7mm]{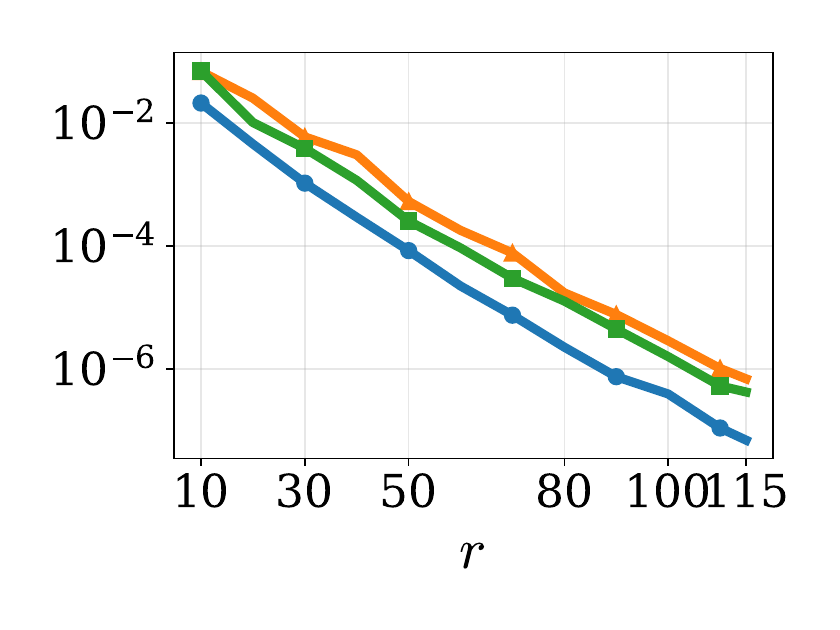}
    \caption{FH chain, abs.~error.}
    \label{fig:num_fixed_rank_proj_abs_fh}
  \end{subfigure}
  \hfill
  \begin{subfigure}{0.24\textwidth}
    \centering
    \includegraphics[width=\textwidth, trim = 7mm 12mm 7mm 7mm]{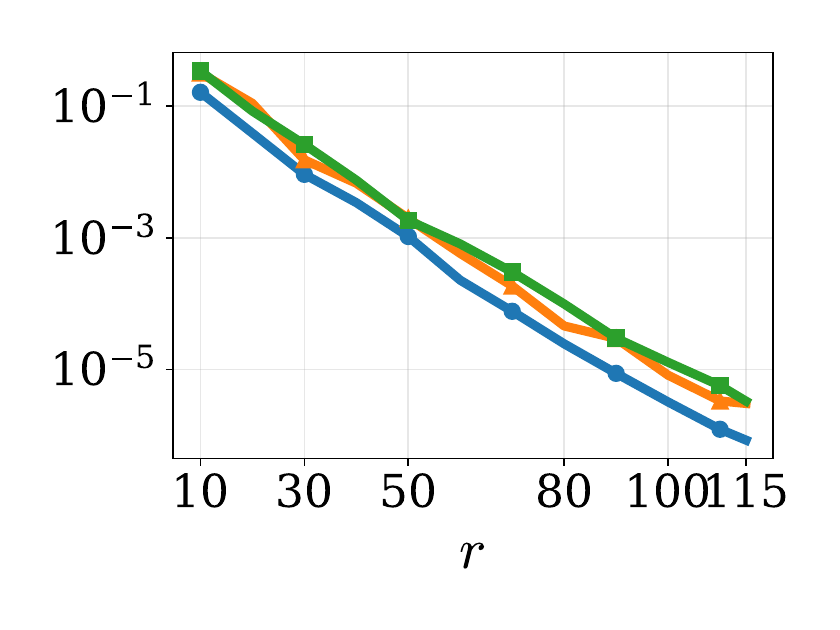}
    \caption{FH chain, rel.~error.}
    \label{fig:num_fixed_rank_proj_rel_fh}
  \end{subfigure}
  \hfill
  \begin{subfigure}{0.24\textwidth}
    \centering
    \includegraphics[width=\textwidth, trim = 7mm 12mm 7mm 7mm]{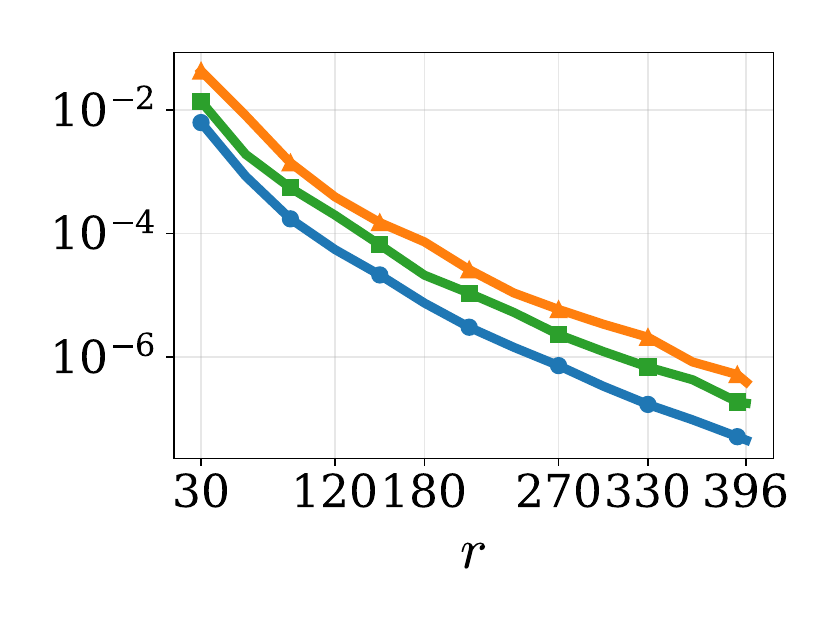}
    \caption{XYZ chain, abs.~error.}
    \label{fig:num_fixed_rank_proj_abs_xyz}
  \end{subfigure}
  \hfill
  \begin{subfigure}{0.24\textwidth}
    \centering
    \includegraphics[width=\textwidth, trim = 7mm 12mm 7mm 7mm]{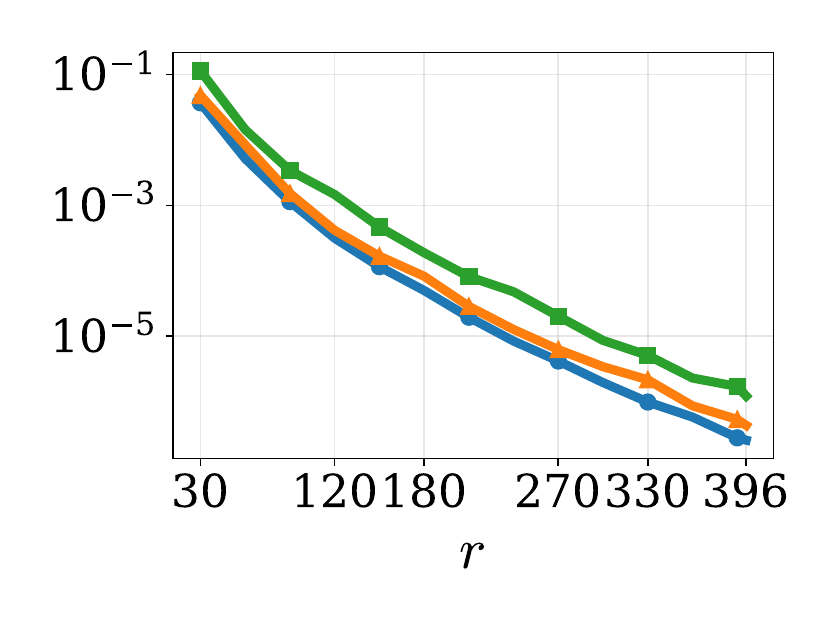}
    \caption{XYZ chain, rel.~error.}
    \label{fig:num_fixed_rank_proj_rel_xyz}
  \end{subfigure}
  \caption{Maximal projection errors $\max_{\prmtr \in \Xi_{\mathrm{train}}}\dist(\SteadyState(\prmtr), V_r)$ on the training grids for
  $L=8$, evaluated at common {RB dimensions} $r$.  The blue, orange, and
  green lines correspond to POD, EigGreedy, and AffGreedy, respectively.}
  \label{fig:num_fixed_rank_projection}
\end{figure}

\subsubsection{Accuracy at fixed {RB dimension}}
We next compare the online approximations at the same {RB dimensions}.
For POD, the online state is computed with the affine residual formulation.
For the greedy spaces, we use the corresponding online reduced solves.

\Cref{fig:num_fixed_rank_accuracy} reports the maximal absolute and relative
approximation errors on the validation grid.
This assesses the online accuracy at parameters not used to construct the
reduced spaces.
Both errors decrease systematically with the {RB dimension}, with POD
providing the most accurate approximation at a prescribed {RB dimension}.
Among the two greedy methods, AffGreedy gives the smaller approximation error
in both measures over nearly the entire range of {RB dimensions}.
The training-grid results show the same ordering and closely follow the
validation curves.
Thus, the advantage of EigGreedy in relative projection error does not carry
over to the approximation error.
The affine formulation enforces unit trace within the reduced solve, whereas
the eigenvalue formulation normalizes the trace afterwards; this additional
normalization may contribute to the larger approximation error of the latter.
At the largest common  {RB dimensions}, the maximal absolute errors of all three methods
are below $4.2\times10^{-5}$ and the maximal relative errors are below
$2.7\times10^{-5}$ for both models.

The attainable accuracy of the residual-driven greedy constructions is also
limited in finite precision.  Their indicators have a quadratic-form
character: they are squared residuals associated with
$\Lindbladian(\prmtr)^\dagger\Lindbladian(\prmtr)$ and, by
\Cref{prp:valid_weak_greedy}, scale like the squared projection error.
Consequently, a projection-error target of a given order requires resolving
indicator values roughly of its square, with the stability factors $\gamma$
and $\beta^2$ imposing a further reduction for EigGreedy and AffGreedy,
respectively.  Hence, reliable comparison over the training grid can break
down before the vector errors reach numerical precision.
POD is less affected during basis construction, 
since it extracts the reduced space from the state snapshots themselves, 
which partly explains its advantage at the highest accuracies considered here.

Although positive semidefiniteness is not enforced, negative eigenvalues of
the reconstructed approximations decrease rapidly with the  {RB dimension}~$r$.
At the final greedy dimensions, the smallest eigenvalues over the validation grids
are at least $1.5\times10^{-9}$ for the FH chain and
$-6.7\times10^{-10}$ for the XYZ chain, so no significant loss of positivity
is observed.

\begin{figure}[htbp]
  \centering
  \begin{subfigure}{0.24\textwidth}
    \centering
    \includegraphics[width=\textwidth, trim = 7mm 12mm 7mm 7mm]{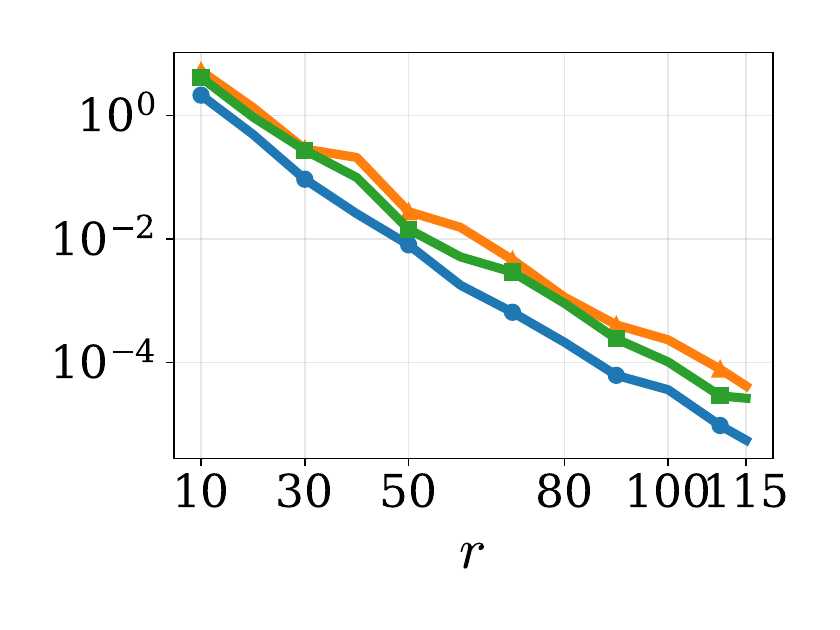}
    \caption{FH chain, abs.~error.}
    \label{fig:num_fixed_rank_abs_fh}
  \end{subfigure}
  \hfill
  \begin{subfigure}{0.24\textwidth}
    \centering
    \includegraphics[width=\textwidth, trim = 7mm 12mm 7mm 7mm]{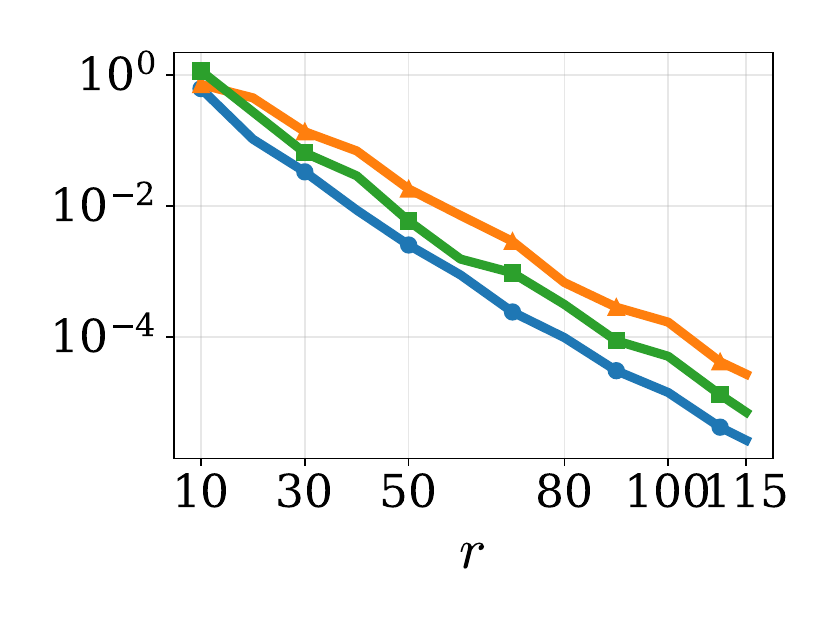}
    \caption{FH chain, rel.~error.}
    \label{fig:num_fixed_rank_rel_fh}
  \end{subfigure}
  \hfill
  \begin{subfigure}{0.24\textwidth}
    \centering
    \includegraphics[width=\textwidth, trim = 7mm 12mm 7mm 7mm]{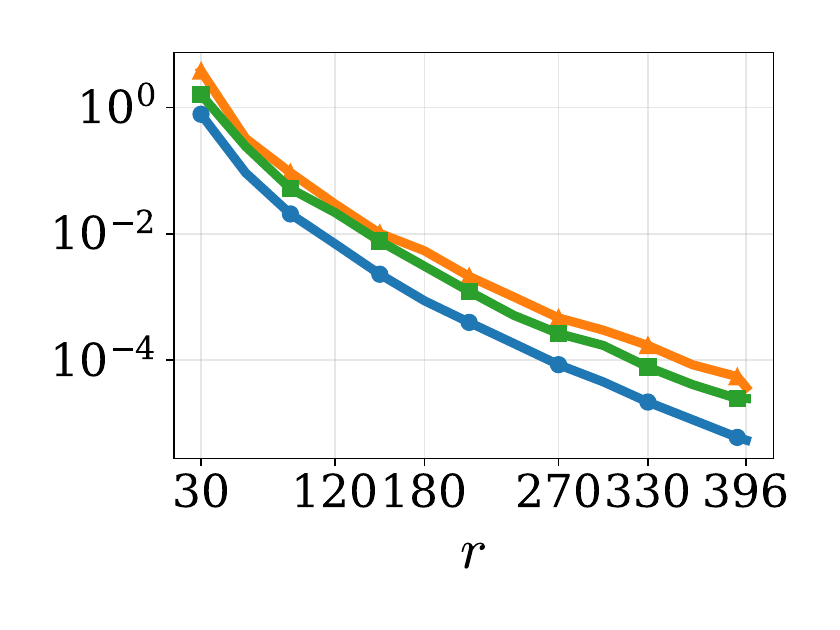}
    \caption{XYZ chain, abs.~error.}
    \label{fig:num_fixed_rank_abs_xyz}
  \end{subfigure}
  \hfill
  \begin{subfigure}{0.24\textwidth}
    \centering
    \includegraphics[width=\textwidth, trim = 7mm 12mm 7mm 7mm]{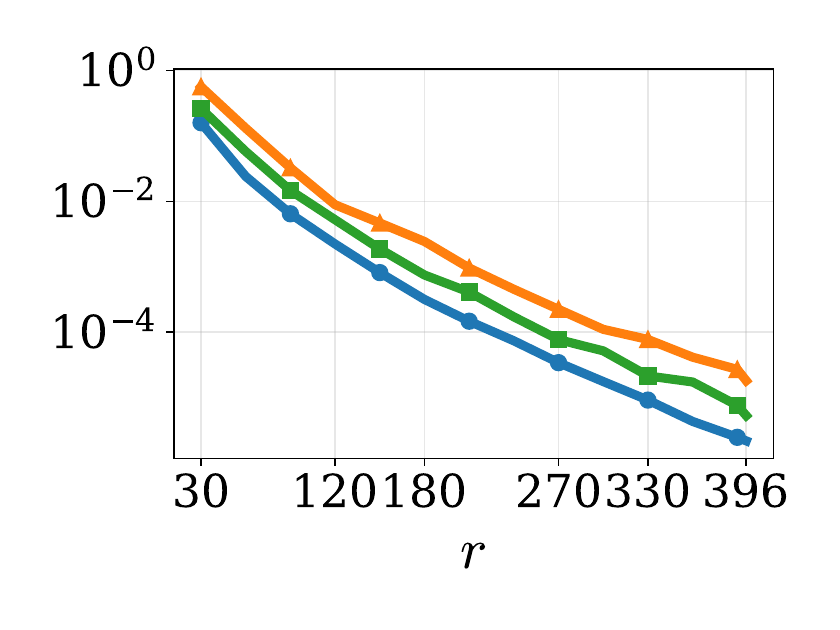}
    \caption{XYZ chain, rel.~error.}
    \label{fig:num_fixed_rank_rel_xyz}
  \end{subfigure}
  \caption{Maximal approximation errors 
  $\max_{\prmtr \in \Xi_{\mathrm{val}}}\|\SteadyState(\prmtr)- \rho^{(r)}_\star(\prmtr)\|$ on the uniform validation grids for
  $L=8$, evaluated at common {RB dimensions} $r$.  
  The blue, orange, and
  green lines correspond to POD, EigGreedy, and AffGreedy, respectively.}
  \label{fig:num_fixed_rank_accuracy}
\end{figure}

\subsection{Greedy indicators and state-error estimates}
\label{subsec:num_error_estimates}

We first examine how the greedy indicators reflect the projection error, 
and then how the residual and the trace-zero stability constant bound the
online state error.

\subsubsection{Indicator effectivity}
The indicator--projection-error equivalences in
\Cref{prp:valid_weak_greedy} motivate the effectivity quotients
$  q_{\star,r}(\prmtr)
  \, \coloneqq \, 
  {\Delta_{\star}^{(r)}(\prmtr)} \big/
       {\dist(\SteadyState(\prmtr),V_{r,\star})^2} $
for $\star\in\{\mathrm{eig},\mathrm{aff}\}$.
By \Cref{prp:valid_weak_greedy}, $q_{\star,r}$ is bounded 
from above and away from zero, 
uniformly in $\prmtr$ and $r$.
We examine this property for $L=8$ on the validation grids.

The left and middle panels of \Cref{fig:num_indicator_effectivity} compare
the maximal indicators and squared projection errors, while the right panels
show the median of $q_{\star,r}$ and its range over the validation grid.
Once the projection error approaches numerical precision, the quotient is no
longer informative.
Hence, we discard parameter points satisfying
\(
\dist(\SteadyState(\prmtr),V_{r,\star})
\leq 10^{-7}\|\SteadyState(\prmtr)\|.
\)

The right panels also relate the observed effectivities to the weakness
parameters $\vartheta$ in the greedy convergence analysis.  By the argument of
\Cref{lem:error_equivalence_weak_greedy}, the square root of the ratio of the
minimum to the maximum of $q_{\star,r}$ is the sampled analogue of the weakness
parameter.  Thus, a narrower effectivity band indicates a stronger greedy
selection rule.

For both models and both greedy procedures, the indicator and squared
projection error have a similar curve form and 
decay at comparable rates over the  {RB dimensions} $r$.
After the initial dimensions, 
the median effectivity depends only weakly on $r$.
For the FH chain, 
it is of order $10 $ for EigGreedy and order $1$ for AffGreedy; 
for the XYZ chain the corresponding orders are $100$ and $10$.
Moreover, AffGreedy gives the narrower effectivity range for both models,
indicating a smaller variation and,  
a larger "intrinsic" weakness parameter $\vartheta$ 
in the sense of \Cref{thm:weak_greedy_convergence_rates}.
The same qualitative behavior is observed for smaller $L$.
Thus, the results support the error equivalences on the sampled parameters.

\begin{figure}[htbp]
  \centering
  \begin{subfigure}{0.32\textwidth}
    \centering
    \includegraphics[width=\textwidth, trim=0 9mm 0 9mm]{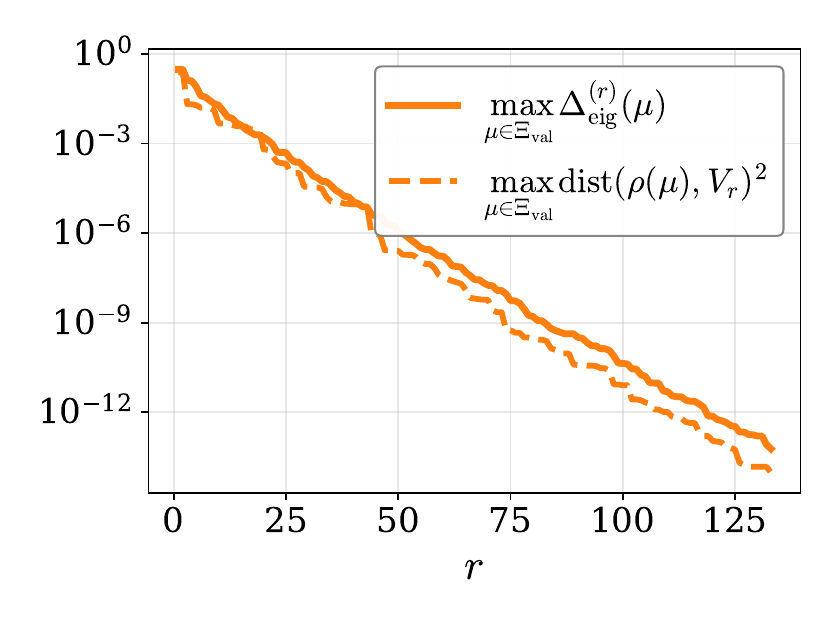}
    \caption{FH, EigGreedy.}
    \label{fig:num_effectivity_eig_fh}
  \end{subfigure}\hfill
  \begin{subfigure}{0.32\textwidth}
    \centering
    \includegraphics[width=\textwidth, trim=0 9mm 0 9mm]{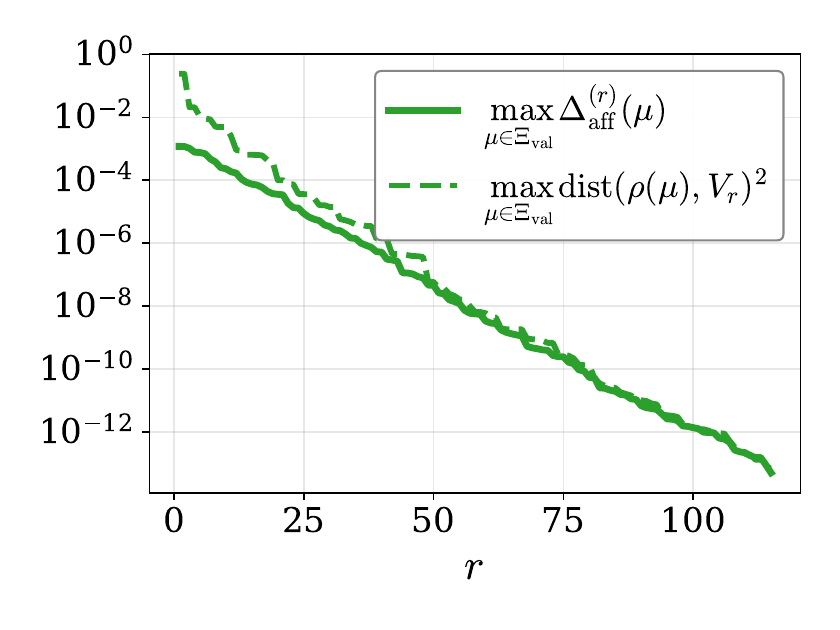}
    \caption{FH, AffGreedy.}
    \label{fig:num_effectivity_aff_fh}
  \end{subfigure}\hfill
  \begin{subfigure}{0.32\textwidth}
    \centering
    \includegraphics[width=\textwidth, trim=0 9mm 0 9mm]{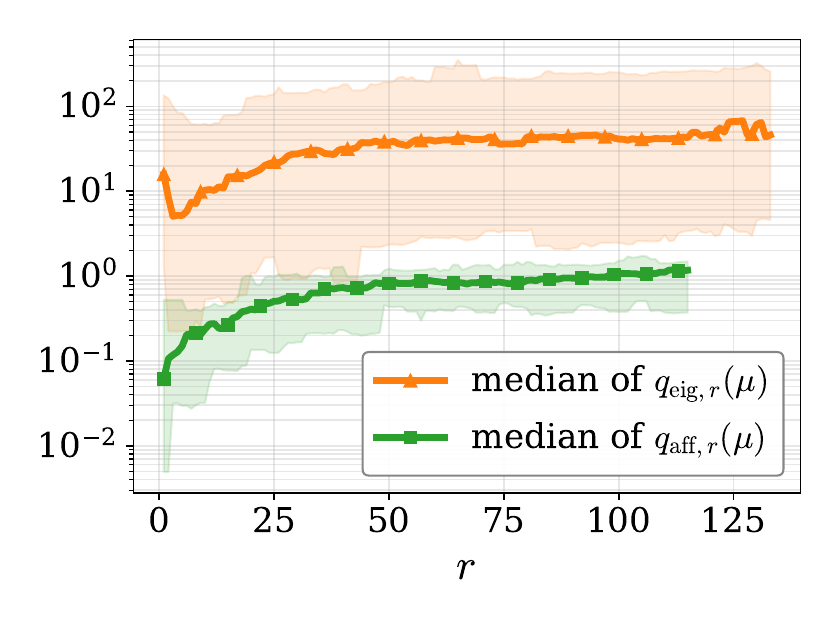}
    \caption{FH effectivity.}
    \label{fig:num_effectivity_band_fh}
  \end{subfigure}
  \vspace{3mm}

  \begin{subfigure}{0.32\textwidth}
    \centering
    \includegraphics[width=\textwidth, trim=0 9mm 0 5mm]{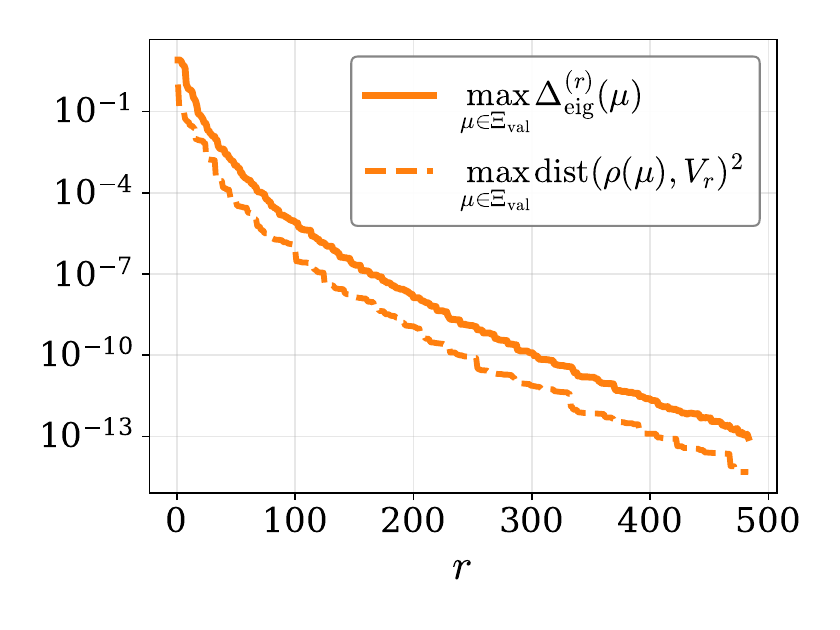}
    \caption{XYZ, EigGreedy.}
    \label{fig:num_effectivity_eig_xyz}
  \end{subfigure}\hfill
  \begin{subfigure}{0.32\textwidth}
    \centering
    \includegraphics[width=\textwidth, trim=0 9mm 0 5mm]{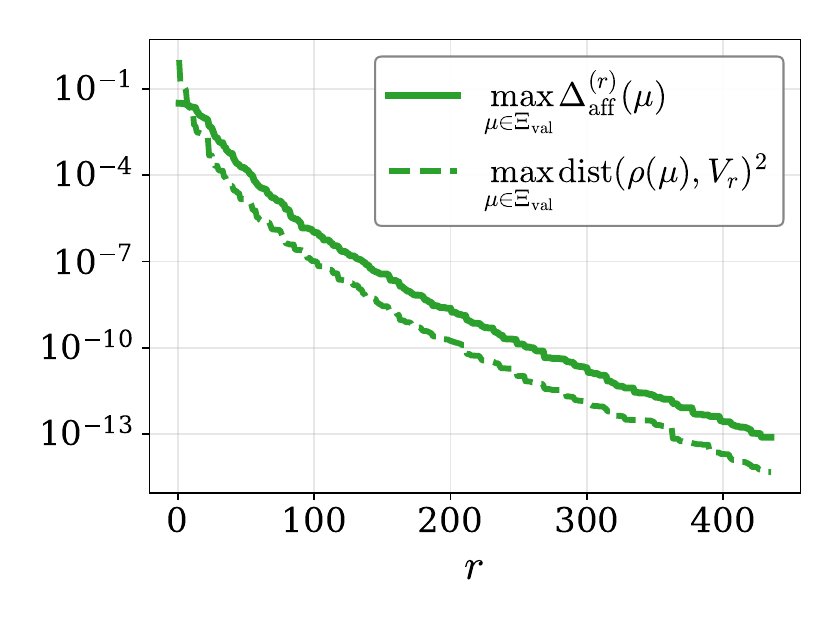}
    \caption{XYZ, AffGreedy.}
    \label{fig:num_effectivity_aff_xyz}
  \end{subfigure}\hfill
  \begin{subfigure}{0.32\textwidth}
    \centering
    \includegraphics[width=\textwidth, trim=0 9mm 0 5mm]{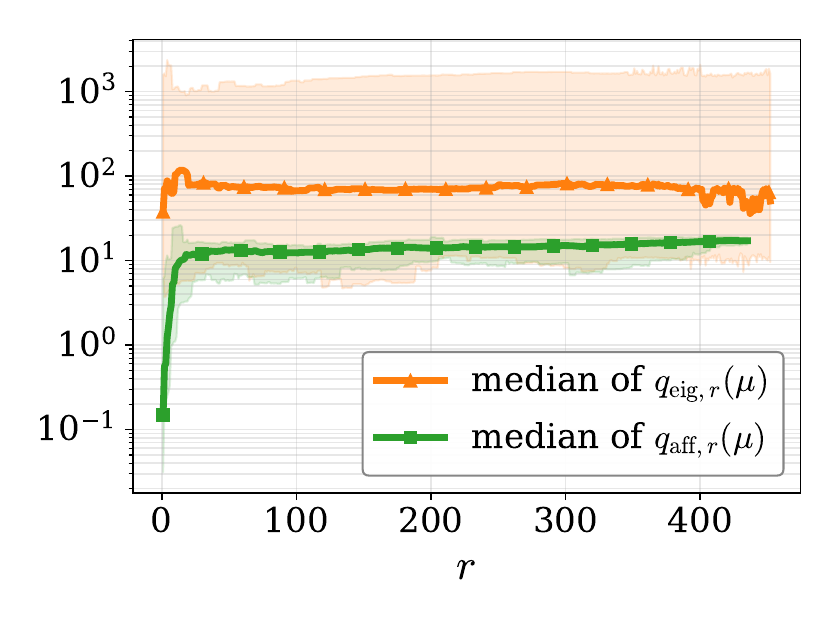}
    \caption{XYZ effectivity.}
    \label{fig:num_effectivity_band_xyz}
  \end{subfigure}
  \caption{Indicator effectivity on the validation grid
  $\Xi_{\mathrm{val}}$ (the $13\times13$ uniform grid) for $L=8$.
  The shaded regions in the right column give its minimum and maximum over the resolvable points.
  }
  \label{fig:num_indicator_effectivity}
\end{figure}

\subsubsection{{Validation of the residual-based state-error bound}}
We next test the state-error estimate from
\Cref{cor:online_approx_steady_state_error} for $L=8$.
At every point of the uniform validation grids, we compute
$\beta(\prmtr)$ and evaluate the local bound
$\|\Lindbladian(\prmtr)\rho_\star^{(r)}(\prmtr)\|/\beta(\prmtr)$.
The left panels of \Cref{fig:num_state_error_bound} show the sampled stability
constant.  Its smallest values are $2.35\times10^{-3}$ for the FH chain and
$7.33\times10^{-2}$ for the XYZ chain.  
The positive, smoothly varying values,
together with \eqref{eq:equivalence_unique_steady_state}, 
provide numerical
evidence for uniqueness of the steady state throughout both parameter domains. 

The right panels compare the true absolute state error with this bound for the
final greedy approximations.  Since the bound is shown on the horizontal axis,
the estimate requires every point to lie on or below the diagonal.  This holds
at all $169$ validation parameters for both greedy methods and also at the two
additional intermediate RB dimensions considered for each method.  At the
final dimensions, the median ratio of the bound to the true error lies between
$97$ and $102$ for the FH chain and is about $13.6$ for the XYZ chain; the
corresponding largest ratios are $2.1\times10^3$ and $45$.  Thus, the
$\beta$-based bound is valid throughout the sampled grids and is markedly
tighter for the XYZ chain.  Compared with the gap--angle lower bound from
\Cref{lem:beta_gap_angle_relation}, using $\beta(\prmtr)$ directly reduces the
median bounds by factors of approximately $1.9$ and $4.8$ for the FH and XYZ
chains, respectively.

\begin{figure}[htbp]
  \centering
  \begin{subfigure}{0.38\textwidth}
    \centering
    \includegraphics[width=\textwidth, trim = 0 13mm 0 12mm]{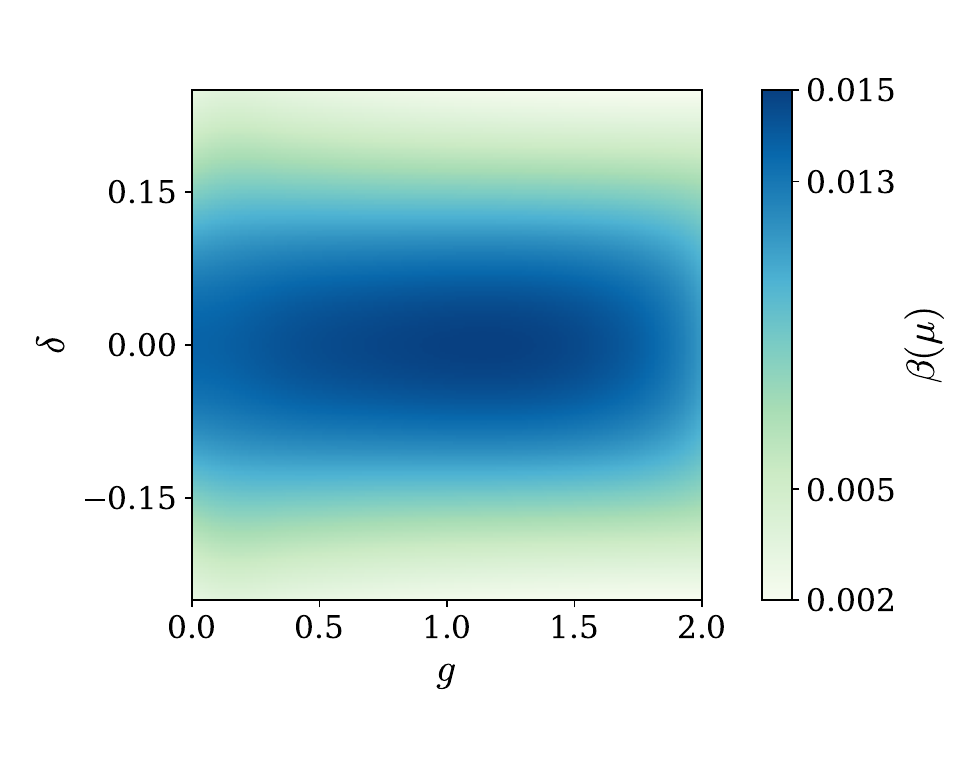}
    \caption{ {FH, sampled $\beta$.}}
    \label{fig:num_cert_beta_fh}
  \end{subfigure}
  \qquad
    \begin{subfigure}{0.32\textwidth}
    \centering
    \includegraphics[width=\textwidth, trim = 0 8mm 0 12mm]{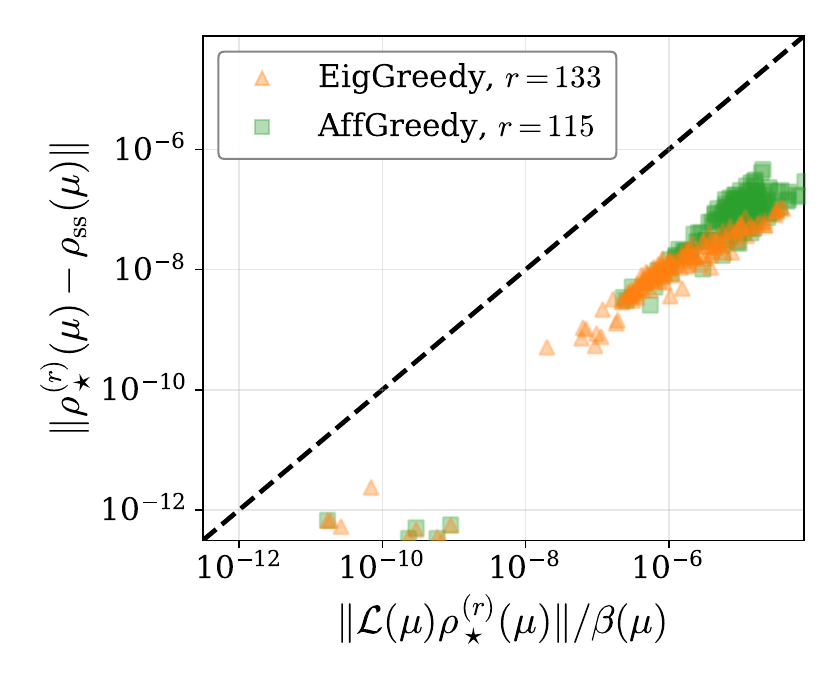}
    \caption{ {FH, error \& upper bound.}}
    \label{fig:num_cert_beta_bound_fh}
  \end{subfigure}

  \vspace{1mm}

  \begin{subfigure}{0.38\textwidth}
    \centering
    \includegraphics[width=\textwidth, trim = 0 13mm 0 5mm]{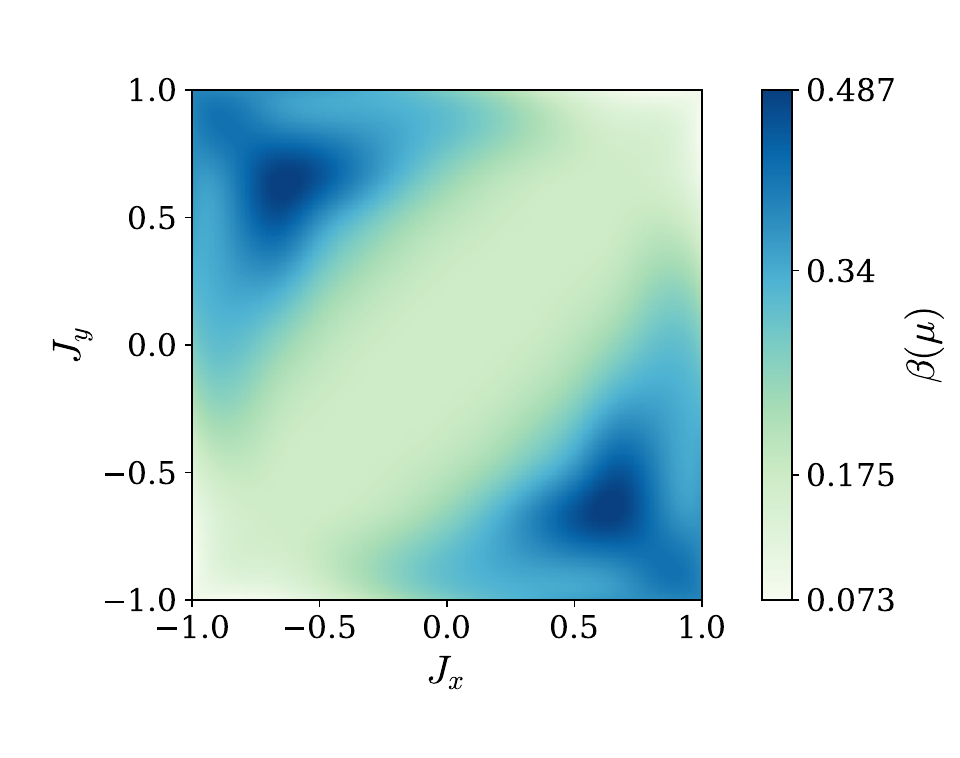}
    \caption{ {XYZ, sampled $\beta$.}}
    \label{fig:num_cert_beta_xyz}
  \end{subfigure}
  \qquad
  \begin{subfigure}{0.32\textwidth}
    \centering
    \includegraphics[width=\textwidth, trim = 0 8mm 0 5mm]{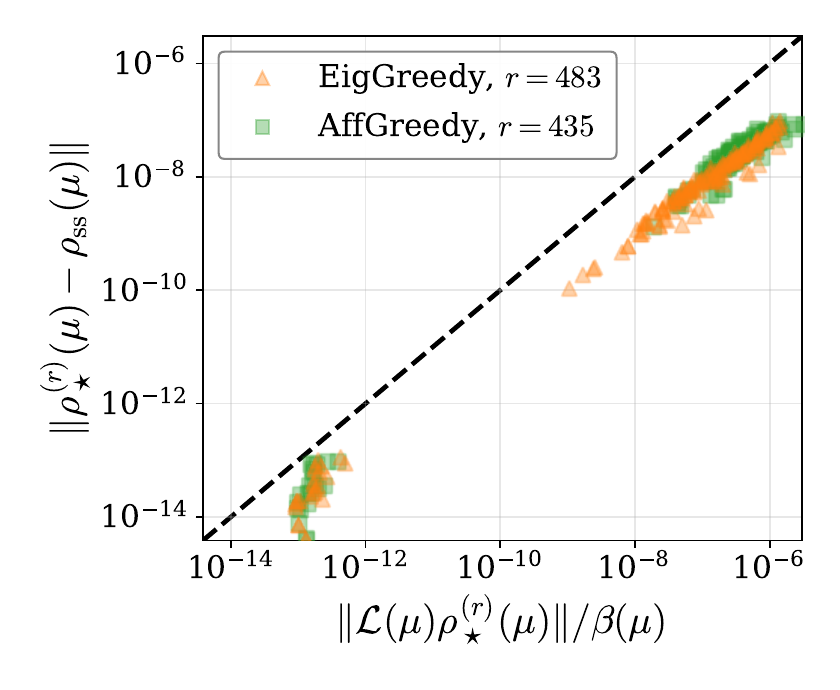}
    \caption{ {XYZ, error \& upper bound.}}
    \label{fig:num_cert_beta_bound_xyz}
  \end{subfigure}
  \caption{Trace-zero stability constants and residual-based state-error bounds
  on the uniform validation grids for $L=8$.  
  The dashed lines in the right panels mark equality.  
  }
  \label{fig:num_state_error_bound}
\end{figure}

\subsection{Online computational performance and physical observables}
\label{subsec:num_performance}

We finally assess the efficiency and accuracy of the online phase at $L=8$,
first for the reconstructed steady states and then for representative physical
observables evaluated directly in the reduced spaces.

\subsubsection{Online cost.}
For each validation parameter, the computation time includes the assembly and
solution of the reduced problem, 
reconstruction of the RB approximation, 
and its symmetrization and trace normalization.
At each validation parameter, the online computation is repeated ten times and
the shortest time is retained.  The reported time is the mean of these
pointwise timings over the validation grid.

\Cref{tab:num_online_cost} reports the resulting mean times and mean pointwise
speedups over the validation grids.
All three methods reconstruct an RB state within a few milliseconds for the
FH chain and within a few tens of milliseconds for the XYZ chain.
Compared with the high-fidelity solves, this gives speedups of order $10^3$
and $10^2$, respectively.

\begin{table}[htbp]
  \centering
  \begin{tabular}{c c c c c}
    \toprule
    Model & Method &  {RB dimension} $r$ & Online time $[\mathrm{ms}]$ & Mean speedup \\
    \midrule
    FH & POD       & $127$ & $2.88$ & $4319$ \\
                   & EigGreedy & $133$ & $3.53$ & $3528$ \\
                   & AffGreedy & $115$ & $2.46$ & $5066$ \\
    \midrule
    XYZ            & POD       & $396$ & $14.1$ & $505$ \\
                   & EigGreedy & $483$ & $25.2$ & $283$ \\
                   & AffGreedy & $435$ & $16.5$ & $431$ \\
    \bottomrule
  \end{tabular}
  \caption{Mean online computation times and mean pointwise speedups on the
  validation grids for the final reduced models at $L=8$.  The online times
  include reconstruction of the RB approximation.
  }
  \label{tab:num_online_cost}
\end{table}

\subsubsection{Reduced evaluation of observables.}
To demonstrate the reduced evaluation of physical observables introduced in \Cref{rem:observable}, 
we choose one local observable for each model.  
For the FH chain, 
we consider the mid-bond current
\(
  J_{\mathrm{mid}}
  \coloneqq
  -\ImgUnit t
  \bigl(c_{j+1}^{\dagger}c_j-c_j^{\dagger}c_{j+1}\bigr),
\)
with $j=L/2$ and the hopping amplitude $t=1$ fixed as in \eqref{eq:num_fh_hamiltonian}.
For the XYZ chain, we take the mid-chain magnetization
$M_{\mathrm{mid}}\coloneqq\sigma_{L/2+1}^{z}$, where
$\sigma_j^z$ is the local Pauli operator defined in
\Cref{subsec:num_setup}.
Their exact expectation values at $\prmtr$ are denoted by 
$\langle J_{\mathrm{mid}}\rangle_\prmtr$ and $\langle M_{\mathrm{mid}}\rangle_\prmtr$, respectively.
The approximated expectation values w.r.t. the RB space $V_r$, 
denoted by $\langle J_{\mathrm{mid}}\rangle_\prmtr^{(r)}$ and $\langle M_{\mathrm{mid}}\rangle_\prmtr^{(r)}$, 
are evaluated directly from the reduced coefficient
vectors as described in \Cref{sec:RBM}, without reconstructing the RB
approximation in the full-order space at any validation point.

\Cref{fig:num_observables} shows the high-fidelity observable maps and the
maximal absolute errors on the validation grids.
Observe that $\langle J_{\mathrm{mid}}\rangle_\prmtr $ 
depends mainly on $|\delta|$, 
whereas $\langle M_{\mathrm{mid}}\rangle_\prmtr $ depends mainly on the
difference $J_x-J_y$.
For all three reduced spaces, the errors decrease by several orders of
magnitude as the  {RB dimension}~$r$ increases.
The maximal error in $\langle J_{\mathrm{mid}}\rangle_{\prmtr}$ is below one percent of
its maximal high-fidelity magnitude at $r=30$, and the corresponding error in
$\langle M_{\mathrm{mid}}\rangle_{\prmtr}$ is below one percent at $r=60$.
At the largest common  {RB dimensions}, the errors are below $4\times10^{-8}$ and
$2\times10^{-7}$, respectively.
For both greedy constructions, the one-time construction of the reduced
observable takes less than $0.07\,\mathrm{s}$, after which each evaluation from
the reduced coefficients takes less than one microsecond.

\begin{figure}[htbp]
  \centering
    \begin{subfigure}{0.32\textwidth}
    \centering
    \includegraphics[width=\textwidth, trim=2mm 9mm 2mm 8mm]{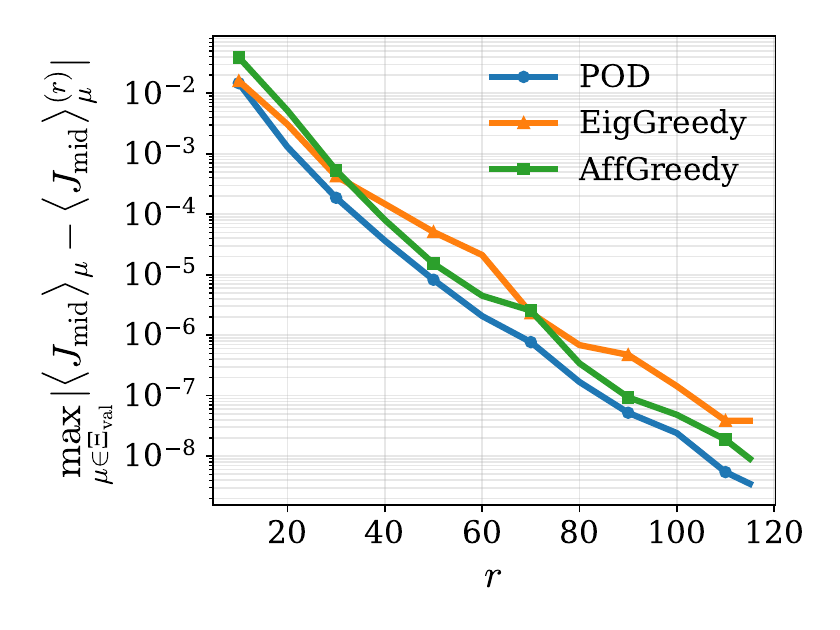}
    \caption{Error in $\langle J_{\mathrm{mid}}\rangle$.}
    \label{fig:num_obserr_fh}
  \end{subfigure}
  \qquad
  \begin{subfigure}{0.32\textwidth}
    \centering
    \includegraphics[width=\textwidth, trim=2mm 9mm 2mm 8mm]{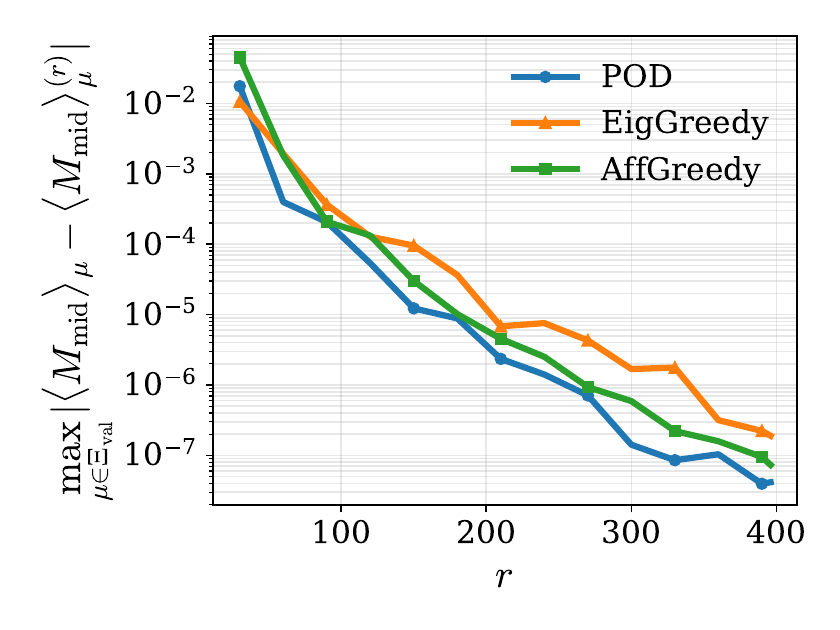}
    \caption{Error in $\langle M_{\mathrm{mid}}\rangle$.}
    \label{fig:num_obserr_xyz}
  \end{subfigure}
  \caption{
  Maximal absolute observable errors 
  on the uniform validation grids for $L=8$.  
  }
  \label{fig:num_observables}
\end{figure}

\section{Conclusion}
\label{sec:conclusion}

In this work, we developed and analyzed an RB framework for the repeated
approximation of parametric Lindbladian steady states; see \Cref{sec:RBM}.
Building on the reduced eigenvalue problem
\eqref{eq:abstract_EVP}, we introduced the trace-constrained residual problem
\eqref{eq:abstract_constraint_res_min} and placed both formulations within a
common residual-driven greedy construction
\Cref{alg:greedy_reduced_basis}.  Both formulations admit an offline--online
decomposition, while positive semidefiniteness can optionally be imposed by
the post-processing procedures from \Cref{subsec:post_processing}.

The analysis exploits the dual interpretation of the steady-state problem as
a parametric eigenvalue and source problem.  Under uniqueness, the
residual-based estimate in
\Cref{cor:online_approx_steady_state_error} controls the state error, while
the indicator equivalences in \Cref{prp:valid_weak_greedy} establish valid
weak-greedy selection for both formulations.  Analytic parameter dependence
then yields generalized-exponential width and greedy convergence; see
\Cref{prop:analytic_width_decay} and \Cref{cor:idealized_greedy_convergence}.  
We also quantify the effect of a finite training set
in \Cref{prop:finite_training_set_convergence}.

The experiments demonstrate compressibility of the solution sets of both tested models
and show comparable sub-exponential decay of the POD spectra and greedy indicators.  The indicators
reliably reflect the projection error, and the residual-based state-error
bounds are valid throughout the sampled parameter grids, although not always
sharp.  The reduced models also provide substantial online speedups and accurate
observable evaluations.  
While POD is generally most accurate at a fixed RB
dimension, AffGreedy avoids computing the full snapshot ensemble, 
and, among the two greedy procedures, 
gives the more favorable balance of offline cost
and approximation accuracy.


\subsection*{Acknowledgments}
Funded by Deutsche Forschungsgemeinschaft (DFG, German Research Foundation) under Germany's Excellence Strategy - EXC 2075 - 390740016. 
We acknowledge the support by the Stuttgart Center for Simulation Science (SimTech).

The authors used generative AI tools during the preparation of this paper for brainstorming,
 theorem revision, improving exposition, numerical implementation and language revision. 
All AI-assisted output was critically reviewed and revised.
The authors take full responsibility for the mathematical arguments and proofs, numerical results, and scientific conclusions
of this work.

\bibliographystyle{macros/plain-doi}
\bibliography{macros/journalabbr,macros/Literature}    

\end{document}